\documentclass[11pt]{amsart}
\usepackage{amscd,amsmath,amsfonts,amssymb,enumerate,amsthm}
\usepackage{graphicx}
\usepackage{wasysym}
\usepackage{ifpdf}
\usepackage{enumitem}
\usepackage[hyperfootnotes=false]{hyperref}
\hypersetup{hidelinks}
\usepackage{setspace}
\usepackage[usenames]{xcolor}
\usepackage{amsrefs}
\usepackage[utf8]{inputenc}
\usepackage[english]{babel}
\usepackage{fullpage}

\newtheorem{prop}{Proposition}[section]

\newtheorem{thm*}{Theorem}

\newtheorem{cor*}[thm*]{Corollary}
\newtheorem{defn}[prop]{Definition}
\newtheorem{rem}[prop]{Remark}

\newtheorem{lem}[prop]{Lemma}
\newcommand{\mref}[1]{(\ref{#1})}

\title{On the Kähler-Einstein metric at strictly pseudoconvex points}
\author{Sébastien Gontard}
\email{Sebastien.Gontard@univ-grenoble-alpes.fr}

\address{Univ. Grenoble Alpes, CNRS, IF, 38000 Grenoble, France}


\begin{document}
\maketitle
\begin{abstract}
We prove a local boundary regularity result for the complete Kähler-Einstein metrics of negative Ricci curvature near strictly pseudoconvex boundary point. We also study the asymptotic behaviour of their holomorphic bisectional curvatures near such points.
\end{abstract}

\section*{Introduction}
In 1980, S.-Y. Cheng and S.-T. Yau proved that every bounded strictly pseudoconvex  domain $\Omega \subset \mathbb{C}^n$, $n\geq 2$, with boundary of class $\mathcal{C}^7$, admits a complete Kähler-Einstein metric of negative Ricci curvature (for convenience, we will only work with Ricci curvature $=-(n+1)$). Namely, they proved that there exists a (unique) solution $g\in \mathcal{C}^{\omega}\left(\Omega\right)$ to the Monge-Ampère equation
\begin{equation}
\label{MAOri}
Det\left(g_{i\bar{j}}\right)=e^{(n+1)g} \quad \text{on $\Omega$},
\end{equation}
satisfying the following boundary condition: 
\begin{equation}
\label{BVal}
g=+\infty \quad \text{on $\partial \Omega$}.
\end{equation}
\\By comparing this solution to the approximate solutions constructed by C. Fefferman in \cite{Fef2}, they proved that if $\Omega$ is bounded, strictly pseudoconvex with boundary of class $\mathcal{C}^{\max\left(2n+9,3n+6\right)}$, then $e^{-g}\in \mathcal{C}^{n+1+\delta}\left(\overline{\Omega}\right)$ for every $\displaystyle \delta \in \left[0,\frac{1}{2}\right[$, and the holomorphic sectional curvatures of this metric tend to -2, which are the curvatures of the unit ball equipped with its Bergman-Einstein metric, at any boundary point. Note that if the boundary is of class $\mathcal{C}^\infty$, J. Lee and R. Melrose proved that $e^{-g}\in \mathcal{C}^{n+1+\delta}\left(\overline{\Omega}\right)$ for every $\displaystyle \delta \in \left[0,1\right[$ (see \cite{LMbobe}), and this regularity is optimal in general.
\\We prove a local version of the result of S.-Y. Cheng and S.-T. Yau. Namely, we prove the following theorem: 

\begin{thm*}
\label{CYLoc}
Let $\Omega\subset \mathbb{C}^n$, $n\geq 2$, and $q\in \partial \Omega$. Assume that there exists a neighborhood of $q$ on which $\partial \Omega$ is strictly pseudoconvex and of class $\mathcal{C}^k$ with $k\geq \max(2n+9,3n+6)$.
Moreover, assume that $\Omega$ carries a complete Kähler-Einstein metric induced by a function $g$ that satisfies conditions \mref{MAOri} and \mref{BVal}. Then there exists an open set $U\subset \mathbb{C}^n$ containing $q$ such that for every $\displaystyle \delta \in \left[0,\frac{1}{2}\right[$, we have: 
$$ e^{-g}\in \mathcal{C}^{n+1+\delta}\left(\overline{\Omega\cap U}\right).$$
\end{thm*}

Note that J. Bland already obtained this result in the case of ``nice" strictly pseudoconvex boundary points, and also obtained that $e^{-g} \in \mathcal{C}^{\frac{n}{2}+\delta}\left(\overline{\Omega\cap U}\right)$ in the general case (see \cite{Bla1}).
Regarding the curvature behavior, we prove the following: 

\begin{thm*}
\label{Main}
Let $\Omega\subset \mathbb{C}^n$, $n\geq 2$, and $q\in \partial \Omega$. Assume that there exists a neighborhood of $q$ on which $\partial \Omega$ is strictly pseudoconvex and of class $\mathcal{C}^k$ with $k\geq \max\left(2n+9,3n+6\right)$.
Moreover, assume that $\Omega$ carries a complete Kähler-Einstein metric induced by a function $g$ that satisfies conditions \mref{MAOri} and \mref{BVal}. Then,
\begin{equation}
\label{AsyBis}
\sup_{v,w\in S(0,1)}\left(Bis_{g,z}(v,w)+\left(1+\frac{\left\lvert \langle v;w\rangle_{g,z}\right \rvert^2}{\langle v;v\rangle_{g,z}^2\langle w;w\rangle_{g,z}^2}\right)\right) \underset{z \to q}{\longrightarrow }0.
\end{equation}
\end{thm*}

$\newline$
Here and from now on, $Bis_{g,z}(v,w)$ (respectively $\langle v;w\rangle_{g,z}$) stands for the holomorphic bisectional curvature (respectively the Hermitian scalar product) of the Kähler metric induced by the potential $g$, at point $z$, between the directions $v$ and $w$ (a more precise definition is given in Section 1).

Especially, using the results on the existence of Kähler-Einstein metrics of N. Mok and S.-T. Yau in the case of bounded pseudoconvex domains (see \cite{MY}), and of A. Isaev in the case of pseudoconvex tube domains with unbounded base (see \cite{Isa1}), we directly deduce:

\begin{cor*}

Let $n\geq2$. Let $\Omega\subset \mathbb{C}^n$ be either a bounded pseudoconvex domain with boundary of class $\mathcal{C}^2$, or a tube domain whose base is convex and does not contain any straight line. Let $g\in \mathcal{C}^\omega\left(\Omega\right)$ be the Kähler-Einstein potential on $\Omega$ that satisfies conditions \mref{MAOri} and \mref{BVal}. Let $q\in \partial \Omega$. Assume that there exists a neighborhood of $q$ on which $\partial \Omega$ is strictly pseudoconvex and of class $\mathcal{C}^k$ with $k\geq \max(2n+9,3n+6)$.
Then there exists an open set $U\subset \mathbb{C}^n$ containing $q$ such that we have: 
$$\forall \delta \in \left[0,\frac{1}{2}\right[,\quad e^{-g}\in \mathcal{C}^{n+1+\delta}\left(\overline{\Omega\cap U}\right).$$
Moreover, we have the following curvature behaviour: 
$$\sup_{v,w\in S(0,1)}\left(Bis_{g,z}(v,w)+\left(1+\frac{\left\lvert \langle v;w\rangle_{g,z}\right \rvert^2}{\langle v;v\rangle_{g,z}^2\langle w;w\rangle_{g,z}^2}\right)\right) \underset{z \to q}{\longrightarrow }0.$$
\end{cor*}

The curvature behavior \mref{AsyBis} can also be obtained in pseudoconvex domains, at boundary points for which the squeezing tends to one (for precise reference about the squeezing function, see for instance \cite{Yeu}). Namely: 

\begin{thm*}
\label{BiSqueez}
Let $\Omega\subset \mathbb{C}^n$ be a pseudoconvex domain, $n\geq 2$, and $q\in \partial \Omega$. Assume that the squeezing function of $\Omega$ tends to one at $q$.
Moreover, assume that $\Omega$ carries a complete Kähler-Einstein metric induced by a function $g$ solving equation \mref{MAOri} with condition \mref{BVal} on $\Omega$. Then,
$$\sup_{v,w\in S(0,1)}\left(Bis_{g,z}(v,w)+\left(1+\frac{\left\lvert \langle v;w\rangle_{g,z}\right \rvert^2}{\langle v;v\rangle_{g,z}^2\langle w;w\rangle_{g,z}^2}\right)\right) \underset{z \to q}{\longrightarrow }0.$$
\end{thm*}

$\newline$
We refer the reader to the proof of Theorem \ref{BiSqueez} for a more precise statement in terms of the squeezing function of the domain. In comparison with Theorems \ref{CYLoc} and \ref{Main}, Theorem \ref{BiSqueez} requires neither regularity assumptions on the boundary of the domain nor the strict pseudoconvexity at $q$, but gives no boundary regularity for the Kähler-Einstein potential. However, it is difficult to find
 geometric condition ensuring that the squeezing tends to one at a boundary point of a given domain.
\\We can apply Theorem \ref{BiSqueez} at $\mathcal{C}^2$ strictly pseudoconvex boundary points of a domain admitting a Stein neighborhood basis (see \cite{JooKi}), at $\mathcal{C}^2$ strictly convex boundary points of bounded domains (see \cite{KZ}), but also at every boundary point of the Fornaess-Wold domain, which is convex but not strictly pseudoconvex and has a boundary of class $\mathcal{C}^2$ (see \cite{FoWo}).

This paper is organized as follows. In Section 1, we introduce some notations and formulas that will be used in the other sections. In Section 2, we recall the construction of asymptotically Kähler-Einstein metrics developped by C. Fefferman in \cite{Fef2}. We provide details about the regularity of the functions involved in the construction and on their defining set. In Section 3, we first estimate the norm of the gradient of the difference between the potential of the Kähler-Einstein metric and the potential of an asymptotically Kähler-Einstein metric constructed in Section 2. Then we use these estimates to improve the $\mathcal{C}^0$ estimate. We also derive the higher order estimates, and we use these to prove Theorems \ref{CYLoc} and \ref{Main} at the end of the Section. In Section 4 we prove Theorem \ref{BiSqueez}.

\vspace{2mm}
{\sl Acknowledgements. I would like to thank Professor S. Fu and Professor J. Bland for their kind hospitality during my visit in their institutions and the fruitful discussions.}

\section{Preliminaries and notations}
Throughout the paper, we use Einstein summation notation.

\subsection{Algebra}

We denote by $\mathcal{M}_n \left(\mathbb{C}\right)$ the set of square matrices of size $n$, with complex coefficients. In this set, we denote by 0 the null matrix and by I the identity matrix.
\\Let $A = \left(A_{i\bar{j}}\right),B = \left(B_{i\bar{j}}\right)\in \mathcal{M}_n \left(\mathbb{C}\right)$, $v=\left(v_i\right) \in \mathbb{C}^n$, $w=\left(w_j\right) \in \mathbb{C}^n$.
\\If $A$ is invertible, we note $\left(A^{i\bar{j}}\right)=A^{-1}$. It is characterized by the relations $A^{i\bar{k}}A_{k\bar{j}}=A_{i\bar{k}}A^{k\bar{j}}=1$ if $i=j$, 0 otherwise. Especially, $Tr\left(A^{-1}B\right)=A^{i\bar{j}}B_{j\bar{i}}$, where $Tr$ denotes the trace function. We denote by $Det\left(A\right)$ the determinant of $A$. We denote by $^tA=\left(A_{j\bar{i}}\right)$ the transpose matrix of $A$, and by $\overline{A}=\left(\overline{A_{i\bar{j}}}\right)$ its conjugate. 
\\We denote by $\displaystyle \mathcal{H}_n:=\{A\in \mathcal{M}_n\left(\mathbb{C}\right) / ^tA=\overline{A} \}$ the space of Hermitian matrices of order $n$. If $A \in \mathcal{H}_n$, we note $\langle v ; w\rangle  _A:=A_{i\bar{j}}v_i \overline{w_j}$. Recall that $\langle v ; v\rangle  _A \in \mathbb{R}$.
\\If $A,B\in \mathcal{H}_n$, we define the following relations: 
$$\displaystyle B\geq A \iff \forall v\in \mathbb{C}^n\setminus\{0\},\quad \langle v ; v\rangle  _B\geq \langle v ; v\rangle  _A , \quad B> A \iff \forall v\in \mathbb{C}^n\setminus\{0\},\quad \langle v ; v\rangle  _B> \langle v ; v\rangle  _A.$$
We note $\mathcal{H}_n^+:=\{M\in \mathcal{H}_n / M\geq 0\}$ and $\mathcal{H}_n^{++}:=\{M\in \mathcal{H}_n / M> 0\}$. If $A\in \mathcal{H}_n^+$, we note $\displaystyle \left \lvert v \right \rvert _A:=\langle v; v \rangle_A^\frac{1}{2}$.
\\We will need the following facts that we do not prove: 

\begin{prop}

\label{BHP}
\begin{enumerate}[align=left, leftmargin=*, noitemsep] \item Let $\displaystyle A \in \mathcal{H}_n^+$.Then there exists $R\in \mathcal{H}_n^+$ such that $R^2=A$. The matrix $R$ is called a square root of $A$.
\item Let $A\in \mathcal{H}_n^+$. Then $\displaystyle 0 \leq A \leq Tr\left(A\right)I$.
\item Let $\displaystyle A\in \mathcal{H}_n^{++}$. Then there exist $0<\lambda\leq \Lambda$ such that $\lambda I \leq A \leq \Lambda I$.
\end{enumerate}
\end{prop}

\subsection{Functions and Kähler geometry in open sets of $\mathbb{C}^n$}
We work with the usual topology on $\mathbb{C}^n$, induced by the usual Euclidean norm, that we note $\lvert\cdot \rvert$. For $p\in \mathbb{C}^n$ and $r>0$, we note $S(p,r)$, respectively $B(p,r)$ the Euclidean sphere, respectively the Euclidean open ball, of center $p$ and radius $r$.
\\Let $U \subset \mathbb{C}^n$ be a non-empty open set, and let $z\in U$.
Let $k\in \mathbb{N}$, and let $\epsilon\in [0,1]$. 

\subsubsection{Functions}
We denote by $\mathcal{C}^{k+\epsilon}\left(U\right)$ the set of real valued functions having derivatives up to order $k$ and such that all these derivatives are Hölder of exponent $\epsilon$, and by $\mathcal{C}^\omega\left(U\right)$ the set of analytic functions in $U$. We simply note $\mathcal{C}^{k}\left(U\right):=\mathcal{C}^{k+0}\left(U\right)$.
\\A function $f\in \mathcal{C}^{k+\epsilon}\left(U\right)$ is in $\mathcal{C}^{k+\epsilon}\left(\overline{U}\right)$ if all its derivatives up to order $k$ extend continuously to the closure $\overline{U}$ of $U$.
\\If $f\in \mathcal{C}^{k+\epsilon}\left(U\right)$ and $\left(i_1,j_1,\dots,i_n,j_n\right) \in \mathbb{N}^{2n}$ satisfies $s:=\sum_{l=1}^n(i_l+j_l) \leq k$, we denote by $\displaystyle f_{i_1\overline{j_1}\dots i_n\overline{j_n}}:=\frac{\partial^s f}{\partial z_1^{i_1}\partial\overline{z_1}^{j_1}\dots \partial z_n^{i_n}\partial\overline{z_n}^{j_n}}$. In particular, this notation is consistent with the notation of complex matrices introduced above. Also, observe that if $f\in \mathcal{C}^1\left(U\right)$, then for every $1\leq j \leq n$, we have $f_{\bar{j}}=\overline{f_j}$.
\\A function $f \in \mathcal{C}^2\left(U\right)$ is plurisubharmonic at $z$, respectively strictly plurisubharmonic at $z$, if $\left(f_{i\bar{j}}(z)\right)\geq 0$, respectively $\left(f_{i\bar{j}}(z)\right)> 0$. A function $f\in \mathcal{C}^2\left(U\right)$ is (strictly) plurisubharmonic in $U$ if it is (strictly) plurisubharmonic at every point of $U$. 
\\We will need the following fact that we do not prove:
\begin{prop}
\label{PSHK}
Let $U\subset \mathbb{C}^n$ be an open bounded set and let $f \in \mathcal{C}^2\left(\overline{U}\right)$ be a strictly plurisubharmonic function. Then there exist constants $0<\lambda\leq \Lambda$ such that $\lambda I \leq \left(f_{i\bar{j}}\right) \leq \Lambda I$ on $\overline{U}$.
\end{prop}

\subsubsection{Kähler metrics}
A Kähler metric in $U$ is an element of $\mathcal{C}\left(U,\mathcal{H}_n^{++}\right)$, that is, a matrix $\left(g_{i\bar{j}}\right)$ with continuous coefficients in $U$ and such that for every $z\in U$, $\left(g_{i\bar{j}}(z)\right)\in \mathcal{H}_n^{++}$.
\\We say that a Kähler metric $\left(g_{i\bar{j}}\right)$ is induced by a function $u\in \mathcal{C}^2\left(U\right)$, called a (Kähler) potential for $\left(g_{i\bar{j}}\right)$, if $\left(u_{i\bar{j}}\right)=\left(g_{i\bar{j}}\right)$ in $U$.
\\If $v=\left(v_i\right) \in \mathbb{C}^n$, $w=\left(w_j\right) \in \mathbb{C}^n$, $f,g\in \mathcal{C}^2\left(U\right)$ and $g$ is a Kähler potential in $U$, we define the following quantities: 
$\newline$ $\bullet$
 $\langle v;w\rangle_{g}:=\langle v;w\rangle_{\left(g_{i\bar{j}}\right)}=g_{i\bar{j}}v_i\overline{w_j}$: the scalar product of $v$ and $w$, for the metric $\left(g_{i\bar{j}}\right)$.
$\newline$ $\bullet$ $\displaystyle \lvert v \rvert _g:=\langle v;v\rangle_{g}^\frac{1}{2}$: the norm of $v$ for $\left(g_{i\bar{j}}\right)$.
$\newline$ $\bullet$ $\displaystyle Ric(g):=-Log\left(Det\left(g_{i\bar{j}}\right)\right)$: the Ricci form of $\left(g_{i\bar{j}}\right)$.
$\newline$ $\bullet$ $\displaystyle \left\lvert \nabla f \right \rvert_g:=\lvert \left(f_i\right) \rvert _{\left(g^{i\bar{j}}\right)}=\left(g^{i\bar{j}}f_if_{\bar{j}}\right)^\frac{1}{2}$: the norm of the complex gradient of $f$ for $\left(g_{i\bar{j}}\right)$.
$\newline$ $\bullet$ $ \displaystyle \Delta_g f:=Tr\left(\left(g^{i\bar{j}}\right)\left(f_{i\bar{j}}\right)\right)=g^{i\bar{j}}f_{j\bar{i}}$: the Laplacian of $f$ for $\left(g_{i\bar{j}}\right)$.
$\newline$
\\Moreover, if $g\in\mathcal{C}^4\left(U\right)$ and $v,w\neq 0$, we also define: 
$\newline$ $\bullet$ $ \displaystyle \forall 1\leq i,j,k,l \leq n,\quad R_{i\bar{j}k\bar{l}}(g):=-g_{i\bar{j}k\bar{l}} +\sum_{1\leq p,q \leq n}g_{ik\bar{p}}g^{\bar{p}q}g_{q\bar{j}\bar{l}}$: the curvature coefficients of $\left(g_{i\bar{j}}\right)$.
$\newline$ $\bullet$ $ \displaystyle Bis_{g}(v,w):=\displaystyle \frac{R_{i\bar{j}k\bar{l}}(g)v_i\overline{v_j}w_k \overline{w_l}}{\lvert v \rvert_{g} ^2 \lvert w \rvert_{g} ^2}$: the holomorphic bisectional curvature of $\left(g_{i\bar{j}}\right)$, between directions $v$ and $w$.
$\newline$ $\bullet$ $H_g(v):=Bis_g(v,v)$: the holomorphic sectional curvature of $\left(g_{i\bar{j}}\right)$, in the direction $v$.
\\If needed, we will specify the point $z$ at which these quantities are computed by using the following notations: $\langle v;w\rangle_{g,z}$, $Ric(g)(z)$, $\left\lvert \nabla f \right \rvert_{g,z}$, $\Delta_g f(z)$, $R_{i\bar{j}k\bar{l}}(g)(z)$, $Bis_{g,z}(v,w)$, etc. 
\\In the special case of the usual metric on $\mathbb{C}^n$, that is to say $\tilde{g}=\lvert \cdot \rvert ^2$ (or, equivalently, $\left(\tilde{g}_{i\bar{j}}\right)=I$), we simply note $\langle v;w\rangle$, respectively $\left\lvert \nabla f \right \rvert$, instead of $\langle v;w\rangle_{\tilde{g}}$, respectively $\left\lvert \nabla f \right \rvert_{\tilde{g}}$. We proceed likewise with the other notations.
\\Recall that the metric induces a distance function, that we denote by $d_g$. 
We say that the metric is complete if the space $\left(U,d_g\right)$ is complete.
\\We say that a Kähler metric induced by a potential $g \in \mathcal{C}^4\left(U\right)$ is Kähler-Einstein if there exists $\lambda \in \mathbb{R}$ such that $\left(Ric(g)_{i\bar{j}}\right)=\lambda\left(g_{i\bar{j}}\right)$. We point out that in this paper, all the involved Kähler-Einstein metrics satisfy $\left(Ric(g)_{i\bar{j}}\right)=-(n+1)\left(g_{i\bar{j}}\right)$.
\\Note that by definition every strictly plurisubharmonic function in $U$ induces a Kähler potential in $U$. There is another way to construct Kähler potentials from strictly plurisubharmonic negative functions in $U$: 

\begin{prop}
\label{-Log(-psi)}
Let $\psi \in \mathcal{C}^2\left(U\right)$ be a negative strictly plurisubharmonic function. Set $g:=-Log\left(-\psi\right)$. Then $g$ is a Kähler potential in $U$, and the following formulas hold in $U$: 
\begin{eqnarray}
&\label{-Log(-psi)Eq}
\left(-\psi\right)\left( g_{i\bar{j}}\right)=\left(\psi_{i \bar{j}}\right) +\left(\frac{\psi_i \psi_{\bar{j}}}{-\psi}\right),
\\ &\label{-Log(-psi)InvEq}
\left(g^{i\bar{j}}\right)=\left(-\psi\right)\left(\psi^{i \bar{j}}\right) -\left(-\psi\right)\frac{\left(\psi^{i \bar{j}}\right)\left(\psi_{i} \psi_{\bar{j}}\right)\left(\psi^{i \bar{j}}\right)}{-\psi +\lvert \nabla \psi \rvert _\psi^2}.
\end{eqnarray}
\end{prop}

\begin{proof}[Proof of Proposition \ref{-Log(-psi)}]
The function $g$ is well defined and of class $\mathcal{C}^2$ in $U$ by construction, and formula \mref{-Log(-psi)Eq} directly comes from the chain rule.
\\Let $R$ be a square root of $\left(\psi_{i \bar{j}}\right)$. Then $R$ is invertible because $Det\left(R\right)^2=Det\left(\psi_{i \bar{j}}\right)\neq 0$. Set $\displaystyle B:=R^{-1}\left(\frac{\psi_i \psi_{\bar{j}}}{-\psi}\right)R^{-1}$ and $A:=\left(-\psi\right)R^{-1}\left(g_{i\bar{j}}\right)R^{-1}=I+B$. Since the rank of $B$ is 1, we have $B^2=Tr\left(B\right)B$. Since $\displaystyle Tr\left(B\right)=\frac{\lvert \nabla \psi \rvert _\psi^2}{-\psi}=\frac{\psi^{i\bar{j}}\psi_i \psi_{\bar{j}}}{-\psi}\geq 0 > -1$, we can do the following computation: 
$$A\left(I-\frac{B}{1+Tr\left(B\right)}\right)=I+\left(\frac{-1}{1+Tr(B)}+1-\frac{Tr(B)}{1+Tr(B)}\right)B=I,$$
and likewise we have $\displaystyle \left(I-\frac{B}{1+Tr\left(B\right)}\right)A=I$. Hence $A$ is invertible, and its inverse is $\displaystyle A^{-1}=\left(I-\frac{B}{1+Tr\left(B\right)}\right)=I-R^{-1}\frac{\left(\psi_{i} \psi_{\bar{j}}\right)}{-\psi +\lvert \nabla \psi \rvert _\psi^2}R^{-1}$. Therefore we obtain the formula \mref{-Log(-psi)InvEq}: 
$$\left( g^{i\bar{j}}\right)=(-\psi)R^{-1}A^{-1}R^{-1}=(-\psi)\left(\psi^{i \bar{j}}\right) -(-\psi)\frac{\left(\psi^{i \bar{j}}\right)\left(\psi_{i} \psi_{\bar{j}}\right)\left(\psi^{i \bar{j}}\right)}{-\psi + \lvert \nabla \psi \rvert _\psi^2}.$$
\end{proof}

$\newline$
Let $\Omega\subset \mathbb{C}^n$ be a domain, let $k\geq 1$ be an integer, and let $q \in \partial \Omega$. We say that $\partial \Omega$ is of class $\mathcal{C}^k$ in a neighborhood of $q$ if there exists a defining function of class $\mathcal{C}^k$ of $\Omega$ in a neighborhood of $q$, that is, an open set $V\subset \mathbb{C}^n$ containing $q$, and a function $\psi \in \mathcal{C}^k\left(V\right)$ satisfying $\Omega \cap V =\{ \psi <0 \}$ and $\forall z \in \partial \Omega \cap V = \{\psi=0\}$, $\lvert \nabla \psi\rvert_z \neq 0$. If $k\geq 2$ and $\partial \Omega$ is of class $\mathcal{C}^k$ in a neighborhood of $q$, we say that $\partial \Omega $ is strictly pseudoconvex in a neighborhood of $q$ if there exists a bounded open set $V\subset \mathbb{C}^n$ containing $q$, and a function $\psi \in \mathcal{C}^k\left(V\right)$ satisfying $\Omega \cap V =\{ \psi <0 \}$, $\forall z \in \partial \Omega \cap V = \{\psi=0\}$, $\lvert \nabla \psi\rvert_z \neq 0$ and $\left(\psi_{i\bar{j}}\right)>0$ in $V$. The function $\psi$ is a strictly plurisubharmonic defining function for $\partial \Omega \cap V$.

The two following results will be needed in Sections 2 and 3: 

\begin{prop}
\label{Grad>0}
Let $\Omega\subset \mathbb{C}^n$ be a domain, let $n\geq 2$ be an integer, and let $q\in \partial \Omega$. Assume that there exists a neighborhood of $q$ on which $\partial \Omega$ is of class $\mathcal{C}^{1}$. Let $V\subset \mathbb{C}^n$ be an open set  containing $q$, let $\psi \in \mathcal{C}^1\left(V\right)$ be a defining function for $\partial \Omega \cap V$. Let $U\subset \overline{U}\subset V$ be a bounded open set containing $q$.
Then, there exists a constant $\epsilon>0$ such that $\displaystyle \inf_{\overline{U}\cap\{\lvert \psi \rvert \leq \epsilon\}} \lvert \nabla \psi \rvert >0.$
\end{prop}

\begin{proof}[Proof of Proposition \ref{Grad>0}]
We argue by contradiction. Then there exists a sequence $\left(z_i\right)_{i\in \mathbb{N}} \in \overline{U}^\mathbb{N}$ such that $\displaystyle \lim_{i \to +\infty}\psi(z_i)=\lim_{i \to +\infty}\lvert \nabla \psi \rvert_{z_i}=0$.
Since $\overline{U}$ is compact, we can assume, up to extracting a subsequence, that $(z_i)_{i\in \mathbb{N}}$ converges in $\overline{U}$. Denote by $z$ its limit. By continuity of $\psi$ at $z$, the condition $\displaystyle \lim_{i \to +\infty}\psi\left(z_i\right) =0$ implies $\psi(z)=0$, which means that $z\in \partial \Omega \cap \overline{U}\subset \partial\Omega \cap V$. On the one hand, it implies that $\lvert \nabla \psi \rvert _z>0$ because $\psi$ is a defining function for $\partial \Omega \cap V$. One the other hand, the continuity of the function $\lvert \nabla \psi \rvert$ at $z$ implies that $\displaystyle \lvert \nabla \psi \rvert_z=\lim_{i\to +\infty} \lvert \nabla \psi \rvert_{z_i}=0$. Hence the contradiction.
\end{proof}

\begin{prop}
\label{EstInv}
Let $\Omega\subset \mathbb{C}^n$ be a domain, and $q\in \partial \Omega$. Assume that there exists a neighborhood of $q$ on which $\partial \Omega$ is strictly pseudoconvex and of class $\mathcal{C}^{2}$. Let $V\subset \mathbb{C}^n$ be a bounded domain containing $q$, $\psi \in \mathcal{C}^2\left(V\right)$ be a strictly plurisubharmonic defining function for $\partial \Omega \cap V$. Let $g:=-Log\left(-\psi \right)$. Then for every bounded open set $U\subset \overline{U}\subset V$ there exist $0<\lambda \leq \Lambda$ such that the following inequalities hold on $\overline{\Omega \cap U}$:
\begin{equation}
\label{EstInvEq}
\lambda\frac{\psi^2}{-\psi +\left \lvert \nabla \psi \right \rvert ^2}I\leq\left(g^{i\bar{j}}\right) \leq \Lambda\left(-\psi\right)I.
\end{equation}
\end{prop}

\begin{proof}[Proof of Proposition \ref{EstInv}]
We use formula \mref{-Log(-psi)InvEq} and notations of Proposition \ref{-Log(-psi)} with $U$ replaced with $\Omega \cap U$. We also use the notations introduced in the proof of Proposition \ref{-Log(-psi)}.
\\According to Proposition \ref{BHP}, we have $\displaystyle \frac{B}{1+Tr(B)} \in \mathcal{H}_n^+$, hence $\displaystyle 0\leq \frac{B}{1+Tr(B)} \leq \frac{Tr(B)}{1+Tr(B)}I$. Since $\displaystyle A^{-1}=I-\frac{B}{1+Tr(B)}$, we deduce $\displaystyle \frac{1}{1+Tr(B)}I=\left(1-\frac{Tr(B)}{1+Tr(B)}\right)I\leq A^{-1}\leq I$. Since $-\psi>0$, we deduce the following: 
$$\frac{-\psi}{-\psi +\left \lvert \nabla \psi \right \rvert_\psi ^2}I\leq\frac{1}{-\psi}R\left(g^{i\bar{j}}\right)R \leq I,$$
$$\frac{\psi^2}{-\psi +\left \lvert \nabla \psi \right \rvert_\psi ^2}\left(\psi^{i \bar{j}}\right)\leq\left(g^{i\bar{j}}\right) \leq \left(-\psi\right)\left(\psi^{i \bar{j}}\right).$$

Moreover, since $\left(\psi^{i \bar{j}}\right)$ is continuous on the compact set $\overline{U}$, there exist $0<\lambda\leq \Lambda$ such that $\lambda I \leq \left(\psi^{i\bar{j}}\right) \leq \Lambda I$ on $\overline{U}$. Hence:
\[\lambda\frac{\psi^2}{-\psi +\left \lvert \nabla \psi \right \rvert_\psi ^2}I\leq\left(g^{i\bar{j}}\right) \leq \Lambda\left(-\psi\right)I.\]
\end{proof}

\section{Construction of local asymptotically Kähler-Einstein metrics}
Let $V$ be an open set. Let $k \geq 2$ be an integer. If $\psi \in \mathcal{C}^{k}(V)$, its Fefferman functional is defined by
$$J(\psi):=(-1)^n Det\left(
\begin{matrix}
&\psi &(\psi_{\bar{j}})
\\ &^t(\psi_i) &(\psi_{i\bar{j}})
\end{matrix}
\right).$$
Then $J(\psi) \in \mathcal{C}^{k-2}(V)$. We observe that
$$J(\psi)=\psi^{n+1}Det\left(\left(-Log(\psi)_{i\bar{j}}\right)\right) \quad \text{on $\{\psi>0\}$},$$
and that the function
$$F:=Log\left(J(\psi)\right)=-Ric\left(-Log(\psi)\right) -\left(-(n+1)Log(\psi)\right)$$
is well defined on $\{\psi>0\}\cap \{Det\left(-Log(\psi)_{i\bar{j}}\right)>0\}$. Especially, if $\left(-Log(\psi)_{i\bar{j}}\right)>0$, $F$ is well defined and measures the defect of $(-Log(\psi))$ to be the potential of a Kähler-Einstein metric: the metric $\left(-Log(\psi)_{i\bar{j}}\right)$ is Kähler-Einstein if and only if $J\left(\psi\right)=1$.
\\Let $\Omega \subset \mathbb{C}^n$ be a domain and $q\in \partial \Omega$. Assume that there exists a neighborhood $V$ of $q$ such that $\partial \Omega \cap V$ is strictly pseudoconvex and of class $\mathcal{C}^k$ with $k\geq 2n+4$. Without loss of generality, we may assume that $V$ is a bounded domain. We describe Fefferman's iterating process in $V$.
\\Let $\varphi \in \mathcal{C}^k(V)$ be a strictly plurisubharmonic defining function for $\partial \Omega \cap V$. Let $\displaystyle U_0:=\{J(-\varphi)>0\}$. Since $\varphi \in \mathcal{C}^2\left(V\right)$ and $J(-\varphi)>0$ on $\partial \Omega \cap V$, the set $U_0$ contains $\partial \Omega \cap V$ and is open. Consider the following constructions on $U_0$: 
$$\varphi^{(1)}:=\frac{\varphi}{J(-\varphi)^{\frac{1}{n+1}}} \quad \text{ and, for $2\leq l \leq n+1$,} \quad \varphi^{(l)}:=\varphi^{(l-1)}\left(1+\frac{1-J(-\varphi^{(l-1)})}{l(n+2-l)} \right).$$
Then, for every $1\leq l\leq n+1$, $\varphi^{(l)}$ is well defined on $U_0$ and $\varphi^{(l)} \in \mathcal{C}^{k-2l}(U_0)$. Moreover, according to the computations done by C. Fefferman in \cite{Fef2}, we have $\displaystyle \frac{J\left(-\varphi^{(l)}\right)-1}{(-\varphi)^l} \in \mathcal{C}^{k-2l-2}\left(U_0\right)$. This ensures that for every integer $1\leq l \leq n+1$, the sets $\displaystyle U_l:=\left\lbrace\left\lvert 1-J(-\varphi^{(l)})\right \rvert< \frac{1}{2}\right\rbrace$ are open and contain $\partial \Omega \cap V$. Consequently, there exist positive constants $r$ and $R$ such that the set
$U:=\left(\cap_{l=0}^{n+1}U_l\right)\cap \left((B(q,R)\cap \partial \Omega)+B(0,r)\right)$
is open, contains $q$, satisfies $\overline{U}\subset V$, and on which every $\varphi^{(l)}$ is a $\mathcal{C}^{k-2l}$ defining function for $\partial \Omega \cap U$. Then according to Proposition \ref{Grad>0}, we can assume (by taking smaller $r$ and $R$ if necessary) that $\displaystyle \min_{1\leq l \leq n+1}\inf_{z\in\overline{U}}\lvert \nabla \varphi^{(l)} \rvert_z>0$ and also $\displaystyle \inf_{z\in \overline{U}}\lvert \nabla \varphi \rvert _z >0$.
\\Since $\partial \Omega \cap V$ is strictly pseudoconvex, we can (by changing $\varphi^{(l)}$ to $\varphi^{(l)}\left(1+t\varphi^{(l)}\right)$ with $t>0$ small and taking smaller $r$ and $R$ if necessary) assume that each $\varphi^{(l)}$ is strictly plurisubharmonic on $\overline{U}$.
\\Finally, the above construction gives, for every $1\leq l \leq n+1$: 
\begin{align*}
\frac{Log\left(J\left(-\varphi^{(l)}\right)\right)}{(-\varphi)^l}&=
\frac{
	Log\left(
	1+\left(J\left(-\varphi^{(l)}\right)-1\right)
		\right)
}
		{(-\varphi)^l}
		\\&=\frac{J\left(-\varphi^{(l)}\right)-1}{(-\varphi)^l}\left(1+\sum_{m=1}^{+\infty}\frac{(-1)^m}{m+1}
		\left(
		J\left(-\varphi^{(l)}
		\right)-1\right)^{m}\right)\in \mathcal{C}^{k-2l-2}\left(\overline{U}\right).
\end{align*}
\\Let us summarize all these facts: 

\begin{prop}
\label{FeffAp}
Let $\Omega \subset \mathbb{C}^n$ be a domain and let $q\in \partial \Omega$. Assume that there exists a neighborhood $V$ of $q$ such that $\partial \Omega \cap V$ is strictly pseudoconvex and of class $\mathcal{C}^k$ with $k\geq 2n+4$.
Then there exists a bounded domain $U$ containing $q$, and a collection of functions $\left(\varphi^{(l)}\right)_{1\leq l \leq n+1}$ satisfying, for every $1\leq l \leq n+1$:
\begin{enumerate}
\item $\varphi^{(l)} \in \mathcal{C}^{k-2l}\left(\overline{U}\right)$,
\vspace{0,1cm}
\item $\Omega\cap \overline{U}=\{\varphi^{(l)}<0\}\cap \overline{U}$,
\vspace{0,1cm}
\item $\displaystyle \inf_{z\in \overline{U}}\lvert \nabla \varphi^{(l)}\rvert _z >0$,
\vspace{0,1cm}
\item $\varphi^{(l)}$ is strictly plurisubharmonic on $\overline{U}$,
\vspace{0,1cm}
\item $\left\lvert 1-J\left(-\varphi^{(l)}\right)\right \rvert \leq \frac{1}{2}$ on $\overline{U}$,
\vspace{0,1cm}
\item $\frac{J\left(-\varphi^{(l)}\right)-1}{(-\varphi)^l} \in \mathcal{C}^{k-2l-2}\left(\overline{U}\right)$,
\vspace{0,1cm}
\item $\frac{\varphi^{(l)}}{\varphi} \in \mathcal{C}^{k-2l}\left(\overline{U}\right)$ and is positive on $\overline{U}$,
\vspace{0,1cm}
\item $\frac{Log\left(J\left(-\varphi^{(l)}\right)\right)}{(-\varphi)^l} \in \mathcal{C}^{k-2l-2}\left(\overline{U}\right)$.
\end{enumerate}
Moreover, we have $\displaystyle \inf_{z\in \overline{U}}\lvert \nabla \varphi \rvert _z >0$.
\end{prop}

\begin{rem}$\bullet$ Especially, conditions $(1)$ to $(4)$ imply that for every integer $1\leq l \leq n+1$, the function $\varphi^{(l)}$ is a strictly plurisubharmonic defining function of $\partial \Omega \cap U$ of class $\mathcal{C}^{k-2l}$.
$\newline$ $\bullet$ If $k\geq 3n+5$, then all the functions $\displaystyle \varphi^{(l)},\frac{J\left(-\varphi^{(l)}\right)-1}{(-\varphi)^l},\frac{\varphi^{(l)}}{\varphi}$ and $\displaystyle \frac{Log\left(J\left(-\varphi^{(l)}\right)\right)}{(-\varphi)^l}$ belong to $\mathcal{C}^{n+1}\left(\overline{U}\right)$. If $k\geq 3n+6$, then all the aforementionned functions belong to $\mathcal{C}^{n+2}\left(\overline{U}\right)\subset \cap _{0\leq \delta \leq 1} \mathcal{C}^{n+1+\delta}\left(\overline{U}\right)$.
$\newline$ $\bullet$ The metrics $\left(-Log\left(-\varphi^{(l)}_{i\bar{j}}\right)\right)$ are called ``asymptotically Kähler-Einstein" on $\partial \Omega \cap \overline{U}$, since they satisfy the condition $J\left(-\varphi^{(l)}\right)(z)\underset{z \to \partial \Omega \cap \overline{U}}{\longrightarrow} 1$ (recall that $\left(-Log\left(-\varphi^{(l)}_{i\bar{j}}\right)\right)$ is Kähler-Einstein on $\Omega \cap \overline{U}$ if and only if $J\left(-\varphi^{(l)}\right)=1$ on $\Omega \cap \overline{U}$).
\end{rem}

\section{Local boundary regularity}

In this Section, we fix an integer $n\geq 2$, a domain $\Omega \subset \mathbb{C}^n$ and a point $q\in \partial \Omega$. 
We assume that $\Omega$ satisfies the hypothesis of Theorem \ref{CYLoc}. Namely, there exists a complete Kähler-Einstein metric induced by a potential $w'\in \mathcal{C}^\omega\left(\Omega\right)$ that satisfies conditions \mref{MAOri} and \mref{BVal}, and there exists a neighborhood $V$ of $q$ such that $\partial \Omega \cap V$ is strictly pseudoconvex and of class $\mathcal{C}^k$ with $k\geq \max\left(2n+9,3n+6\right)$. Thus, we can apply Proposition \ref{FeffAp}, and use the same notations introduced therein.
\\One of the main ideas to prove Theorem \ref{CYLoc} is to compare the complete Kähler-Einstein metric $\left(w'_{i\bar{j}}\right)$ to the aymptotically Kähler-Einstein metrics induced by the strictly plurisubharmonic defining functions $\left(\varphi^{(l)}\right)_{1\leq l \leq n+1}$ as follows.
\\Let $1\leq l \leq n+1$, and set
$$\eta:=\frac{\varphi^{(l)}}{\varphi}, \quad w:=-Log\left(-\varphi^{(l)}\right)=-Log\left(-\eta \varphi\right),\quad F:=Log\left(J\left(-\varphi^{(l)}\right)\right)=Log\left(J\left(-\eta \varphi\right)\right).$$
Then, according to points $(5),(6),(7),(8)$ of Proposition \ref{FeffAp}, $\eta\in \mathcal{C}^{k-2l}\left(\overline{U}\right)$, $w\in \mathcal{C}^{k-2l}\left(\Omega\cap U\right)$, $F\in \mathcal{C}^{k-2l-2}\left(\overline{U}\right)$, $\frac{F}{\left(-\varphi\right)^l}\in \mathcal{C}^{k-2l-2}\left(\overline{U}\right)$ and $w$ and $F$ are related on $\Omega \cap U$ by the condition
\begin{equation}
Det\left(w_{i\bar{j}}\right)=e^{(n+1)w}e^{F}.
\end{equation}
Let $u:=w'-w$. Then, on $\Omega \cap U$, $u$ solves the Monge-Ampère equation
\begin{equation}
\label{MA}
Det\left(w_{i\bar{j}}+u_{i\bar{j}}\right)=e^{(n+1)u-F}Det\left(w_{i\bar{j}}\right).
\end{equation}
Since $w'$ is real analytic in $\Omega$ and $w \in \mathcal{C}^{k-2l}\left(\Omega\cap U\right)$, then $u\in \mathcal{C}^{k-2l}\left(\Omega\cap U\right).$

So, for each integer $1\leq l \leq n+1$, we have an asymptotically Kähler-Einstein metric $\left(w_{i\bar{j}}\right)$ on $\partial\Omega \cap U$, for which the defect of being Kähler-Einstein is encoded in the function $F$, and we study the difference between this metric and the Kähler-Einstein metric $\left(w'_{i\bar{j}}\right)$ on $\Omega \cap \overline{U}$. More precisely, we study the boundary regularity of the difference of their potentials, namely the function $u$.

\subsection{$\mathcal{C}^1$ estimate and consequences}
Whether global (see \cite{CY}) or local (see \cite{Bla1}), the study of the boundary behavior of $u$ relies on its gradient estimate, which relies on the comparison between the metrics $\left(w'_{i\bar{j}}\right)$ and $\left(w_{i\bar{j}}\right)$ (see condition \mref{Lap}). The gradient estimate enables to deduce the boundary behavior of $u$, and then the boundary behavior of the higher order derivatives of $u$ by use of Schauder theory. All these estimates depend on the regularity of the gradient of $\displaystyle \frac{F}{\left(-\varphi\right)^l}$, for which we have the following result: 

\begin{prop}
\label{NabFOK}
Under the hypothesis of Theorem \ref{CYLoc}, and with the notations introduced at the beginning of Section 3, we have $\displaystyle \frac{\left \lvert \nabla F \right \rvert_w ^2}{(-\varphi)^{2l-1}} \in \mathcal{C}^{k-2l-3}\left(\overline{\Omega \cap U}\right)$. In particular, there exists a positive constant $c_\nabla$, such that the following holds on $\Omega \cap U$:
\begin{equation} 
\label{NabF}
\lvert \nabla F \rvert_w ^2 \leq c_\nabla (-\varphi)^{2l-1}.
\end{equation}
\end{prop}

\begin{proof}[Proof of Proposition \ref{NabFOK}]
Let $1\leq i,j \leq n$. Then, according to point $(8)$ of Proposition \ref{FeffAp}, $\displaystyle \frac{F_i}{(-\varphi)^{l-1}}=l\frac{F\varphi_i}{(-\varphi)^{l}}+\varphi\left(\frac{F}{(-\varphi)^l}\right)_i\in \mathcal{C}^{k-2l-3}\left(\overline{U}\right),$ and according to equation \mref{-Log(-psi)InvEq} as well as point $(7)$ of Proposition \ref{FeffAp}, 
$$\displaystyle \frac{w^{i\bar{j}}}{-\varphi}=\frac{\psi}{\varphi}\frac{w^{i\bar{j}}}{-\psi}=\frac{\psi}{\varphi}\left(\psi^{i\bar{j}}+\frac{\left(\left(\psi^{i\bar{j}}\right)\left(\psi_{i}\psi_{\bar{j}}\right)\left(\psi^{i\bar{j}}\right)\right)_{i j}}{-\psi + \lvert \nabla \psi \rvert_\psi ^2}\right)\in \mathcal{C}^{k-2l-2}\left(\overline{\Omega \cap U}\right),$$
 where $\displaystyle \psi:=\varphi^{(l)}$. 
\\Hence 
$\displaystyle \frac{\left \lvert \nabla F \right \rvert_w ^2}{(-\varphi)^{2l-1}}=\frac{w^{i\bar{j}}}{-\varphi}\frac{F_i}{(-\varphi)^{l-1}}\frac{F_{\bar{j}}}{(-\varphi)^{l-1}} \in \mathcal{C}^{k-2l-3}\left(\overline{\Omega \cap U}\right).$
\end{proof}

We improve the gradient estimate obtained in \cite{Bla1} by using the computations of \cite{CY} in a different way. Then we proceed exactly as in \cite{Bla1} to obtain the estimates of the other derivatives of $u$.

\begin{prop}
\label{ControlNab}
Under the hypothesis of Theorem \ref{CYLoc}, and with the notations indroduced at the beginning of Section 3 and in Proposition \ref{NabFOK}, for every $\gamma\in ]0;\min(2n+1,2l-1)[$, there exist positive constants $c$ and $\epsilon$ such that $\lvert \nabla u \rvert_w ^2 \leq c \left(-\varphi\right)^{\gamma}$ on $\Omega \cap U\cap \{\left\lvert \varphi \right \rvert < \epsilon\}$.
\end{prop}

\begin{rem}
\label{m}

$\bullet$
Proposition \ref{ControlNab} improves the results obtained in \cite{Bla1} in the sense that $\partial\Omega \cap U$ is not required to be ``nice".
$\newline$ $\bullet$
Proposition \ref{ControlNab} is a local version of Proposition 6.4 in \cite{CY}.
$\newline$ $\bullet$ The proof of Proposition \ref{ControlNab} will use the fact that $\left \lvert \nabla u \right \rvert_w^2 \in \mathcal{C}^{2}\left(\Omega \cap U\right)$ and is bounded from above, which is true as long as $k\geq 2n+5$ (see page 297 of \cite{Bla1} for further details).
$\newline$ $\bullet$ It will also use the fact that Lemma II in \cite{Bla1} actually works for $\mathcal{C}^{2}$ functions that are bounded below (see Lemma \ref{LMAxP} for a version that fits to our situation).
\end{rem}

\begin{proof}[Proof of Proposition \ref{ControlNab}]
The strategy of the proof of Proposition 6.4 in \cite{CY} is first to show that there exists $\delta_0>0$ such that for every $0<\alpha<n$, $0\leq \beta <n+1$ and $0<\delta \leq \delta_0$ satisfying $\alpha+\beta+\delta \leq 2l-1$, there exist positive constants $\epsilon$ and $c$ such that the following inequality holds on $\Omega \cap \{\left \lvert \varphi \right \rvert \leq \epsilon\}$:
$$\Delta_{w'}\left(\frac{\lvert \nabla u \rvert_w ^2}{(-\varphi)^\beta}-c(-\varphi)^\alpha\right)>\frac{n+1+n\beta-\beta^2}{2} \left( \frac{\lvert \nabla u \rvert_w ^2}{(-\varphi)^\beta}-c(-\varphi)^\alpha \right),$$
and then to apply the generalized maximum principle and choose suitable constants $\alpha$ and $\beta$ to get the conclusion. 
\\In our case, we wish to follow the same strategy when we restrict our considerations to $\Omega \cap U$.
\\We focus our attention on explaining the necessary modifications in the proof of Proposition 6.4 in \cite{CY}, keeping in mind that we look for local estimates in a neighborhood of $\partial \Omega \cap U$. For that purpose, we first explain the dependence of the constants $c_1,\dots,c_9$ with respect to the local data in order to obtain conditions \mref{SPSH} and \mref{PourPMax}. Then we use formulas \mref{SPSH} and \mref{PourPMax} to complete the proof. For each constant, we refer precisely to the condition in \cite{CY} where it is defined.
\\In the sequel, $0<\alpha<n$ and $0\leq \beta <n+1$.
\\$\bullet$ We apply the first Proposition of page 297 in \cite{Bla1} to derive the existence of positive constants $\epsilon$ and $\delta_0$ such that we have the following on $\Omega \cap \overline{U} \cap \{\left \lvert \varphi \right \rvert \leq \epsilon\}$: 
\begin{equation}
\label{Lap}
\left(w'_{i\bar{j}}\right)=\left(1+O\left(\left( -\varphi \right)^{\delta_0} \right ) \right) \left(w_{i\bar{j}}\right),
\end{equation}
which means that there exists a positive constant $c'_1$ such that: 
$$ \left(1-c'_1\left(- \varphi \right) ^{\delta_0} \right)\left(w_{i\bar{j}}\right)\leq \left(w'_{i\bar{j}}\right) \leq \left(1+c'_1\left(- \varphi \right) ^{\delta_0} \right)\left(w_{i\bar{j}}\right).$$
Hence by inverting it we obtain:
$$\left(1+c'_1\left(- \varphi \right)^{\delta_0} \right)^{-1}\left(w^{i\bar{j}}\right)\leq \left(w'^{i\bar{j}}\right) \leq \left(1-c'_1\left(-\varphi \right) ^{\delta_0} \right)^{-1}\left(w^{i\bar{j}}\right).$$
Since $\displaystyle \frac{1}{1-x}=1+\frac{x}{1-x}\leq 1 +2x$ if $ x\in \left[0,\frac{1}{2}\right]$, we have $ \displaystyle  \frac{1}{1-c'_1\left(-\varphi \right ) ^{\delta_0}}\leq 1+2c'_1\left(-\varphi\right)^{\delta_0}$ on the set $\Omega \cap \overline{U} \cap \{\left \lvert \varphi \right \rvert \leq \epsilon\}$ whenever $\epsilon \leq \left(\frac{1}{2c'_1}\right)^\frac{1}{\delta_0}$. 
\\Moreover, since $\displaystyle \frac{1}{1+x}\geq 1-x\geq 1-2x$ for every $ x\in \left[0,\frac{1}{2}\right]$, we also have $\displaystyle  \frac{1}{1+c'_1\left(- \varphi \right )^{\delta_0}}\geq 1-2c'_1\left(-\varphi \right) ^{\delta_0}$ on $\Omega \cap \overline{U} \cap \{\left \lvert \varphi \right \rvert \leq \epsilon\}$. Thus, there exist positive constants $\epsilon$ and $c_1$ such that we have, on $\Omega \cap \overline{U} \cap \{\left \lvert \varphi \right \rvert \leq \epsilon\}$: 
$$\left(1-c_1\left(-\varphi \right)^{\delta_0} \right)\left(w^{i\bar{j}}\right)\leq \left(w'^{i\bar{j}}\right) \leq \left(1+c_1\left(-\varphi \right)^{\delta_0} \right)\left(w^{i\bar{j}}\right).$$
We also take $\epsilon \leq 1$ so that for every $\delta \geq 0$ we have $\lvert \varphi \rvert^\delta \leq 1$. Consequently, we deduce the existence of constants $\epsilon \in ]0,1]$, $\delta_0,c_1>0$ such that for every  $0\leq \delta \leq \delta_0$, we have the following on $\Omega \cap \overline{U} \cap \{\left \lvert \varphi \right \rvert \leq \epsilon\}$: 
\begin{equation}
\label{c_1}
\left(1-c_1\left \lvert \varphi \right \rvert ^{\delta} \right)\left(w^{i\bar{j}}\right)\leq \left(w'^{i\bar{j}}\right) \leq \left(1+c_1\left \lvert \varphi \right \rvert ^{\delta} \right)\left(w^{i\bar{j}}\right).
\end{equation}
This is the same as condition (6.18) in \cite{CY}, except that it holds in a neighborhood of $\partial \Omega \cap \overline{U}$ in our situation (in \cite{CY}, due to the global assumption of strict pseudoconvexity of $\partial \Omega$, the inequalities in \mref{c_1} are valid in a neighborhood of $\partial \Omega$).
\\From now on, we let $\delta \in ]0,\delta_0]$.
\\$\bullet$
The constant $c_2$ (see condition (6.19)) depends only on $c_1$.
\\$\bullet$ 
The constant $c_3$ (see conditions (6.22) and (6.23)) depends only on $c_1$. Especially we have the following on $\Omega \cap \overline{U} \cap \{\left \lvert \varphi \right \rvert \leq \epsilon\}$: 
$$1-c_3\left(-\varphi\right)^\delta \leq \frac{\lvert \nabla\varphi \rvert_{w'} ^2}{\varphi^2} \leq 1+c_3\left(-\varphi\right)^\delta .$$
In our situation we also assume that $\epsilon \leq \left(\frac{1}{2c_3}\right)^\frac{1}{\delta}$, so that we have the following on $\Omega \cap \overline{U} \cap \{\left \lvert \varphi \right \rvert \leq \epsilon\}$: 
\begin{equation}
\label{Nabphi}
\frac{1}{2}\leq 1-c_3\left(-\varphi\right)^\delta \leq \frac{\lvert \nabla\varphi \rvert_{w'} ^2}{\varphi^2} \leq 1+c_3\left(-\varphi\right)^\delta .
\end{equation}
\\$\bullet$ Set $c_4:=2nc_3$ (see condition (6.24)).
\\$\bullet$ According to inequality (6.25), we have, on $\Omega \cap \overline{U} \cap \{\left \lvert \varphi \right \rvert \leq \epsilon\}$: 
$$-\Delta_{w'}(-\varphi)^\alpha\geq \alpha(-\varphi)^\alpha\left[(n-\alpha)\frac{\lvert \nabla\varphi \rvert_{w'} ^2}{\varphi^2}-c_4(-\varphi)^\delta\right].$$
If we assume that $\epsilon< \left(\frac{n-\alpha}{5c_4}\right)^\frac{1}{\delta}$, then we derive the inequality $\frac{(n-\alpha)}{2}\lvert \nabla\varphi \rvert_{w'} ^2-c_4(-\varphi)^{\delta+2}>0$ on $\Omega \cap \overline{U} \cap \{\left \lvert \varphi \right \rvert \leq \epsilon\}$, which leads to the following: 

\begin{equation}
\label{Lapalpha}
-\Delta_{w'}(-\varphi)^\alpha > \frac{\alpha(n-\alpha)}{2}\frac{\lvert \nabla\varphi \rvert_{w'} ^2}{\varphi^2}(-\varphi)^\alpha.
\end{equation}
This is the same as inequality (6.26) in \cite{CY}, but with $c_5=0$.
\\$\bullet$ Set $c_6:=\beta c_4+c_2$ (see condition (6.28)). 
\\$\bullet$ The constant $c_7$ depends only on $c_6$ (see condition (6.29)).
\\ $\bullet$ The constant $c_8$ depends only on $c_3$ and $c_7$ (see condition (6.30)).
\\ $\bullet$ If $\epsilon <\left(\frac{n+1+n\beta -\beta^2}{2c_8}\right)^\frac{1}{\delta}$, then we have, on $\Omega \cap \overline{U} \cap \{\left \lvert \varphi \right \rvert \leq \epsilon\}$: 
$$\frac{n+1+n\beta-\beta^2}{2}-c_8(-\varphi)^\delta >0,$$
so that in our case inequality (6.31) becomes the following:  
\begin{equation}
\label{LapNabu}
\Delta_{w'}\left(\frac{\lvert \nabla u \rvert_w ^2}{(-\varphi)^\beta}\right)>\frac{n+1+n\beta-\beta^2}{2}\frac{\lvert \nabla u \rvert_w ^2}{(-\varphi)^\beta}-\lvert \nabla F \rvert_w ^2 (-\varphi)^{-(\delta + \beta)}.
\end{equation}

Combining \mref{Lapalpha} and \mref{LapNabu}, we obtain, on $\Omega \cap \overline{U} \cap \{\left \lvert \varphi \right \rvert \leq \epsilon\}$ and for every $c>0$: 
$$\Delta_{w'}\left(\frac{\lvert \nabla u \rvert_w ^2}{(-\varphi)^\beta}-c(-\varphi)^\alpha\right)>\frac{n+1+n\beta-\beta^2}{2}\frac{\lvert \nabla u \rvert_w ^2}{(-\varphi)^\beta}-\lvert \nabla F \rvert_w ^2 (-\varphi)^{-(\delta + \beta)} +c\frac{\alpha(n-\alpha)}{2}\frac{\lvert \nabla\varphi \rvert_{w'} ^2}{\varphi^2}(-\varphi)^\alpha.$$
This is exactly the same as inequality (6.31) in \cite{CY}, but with $c_9=0$.
\\ $\bullet$ Using condition \mref{NabF} (Proposition \ref{NabFOK}), we observe that $\lvert \nabla F \rvert_w ^2 \leq c_\nabla (-\varphi)^{\alpha+\delta +\beta}$ whenever $\left\lvert\varphi\right \rvert \leq 1$ and $\alpha + \delta + \beta \leq 2l-1$. 
Therefore, according to \mref{Nabphi}, the following holds on $\Omega \cap \overline{U} \cap \{\left \lvert \varphi \right \rvert \leq \epsilon\}$: 
\begin{align*}
\displaystyle -\lvert \nabla F \rvert_w ^2 (-\varphi)^{-(\delta + \beta)} +c\frac{\alpha(n-\alpha)}{2}\frac{\lvert \nabla\varphi \rvert_{w'} ^2}{\varphi^2}(-\varphi)^\alpha &\geq -c_\nabla (-\varphi)^{\alpha} +c\frac{\alpha(n-\alpha)}{2}\frac{\lvert \nabla\varphi \rvert_{w'} ^2}{\varphi^2}(-\varphi)^\alpha 
\\&\geq \left(-c_\nabla +c\frac{\alpha(n-\alpha)}{4}\right)(-\varphi)^\alpha.
\end{align*}
In particular if we take $c>\frac{4c_\nabla}{\alpha(n-\alpha)}$ the right-hand side is non-negative.
This is exactly what is derived from relation (6.32) in \cite{CY} (see the explanation below relation (6.33) in \cite{CY}), except that in our case it holds on $\Omega \cap \overline{U} \cap \{\left \lvert \varphi \right \rvert \leq \epsilon\}$.

For short, we have proved that there exists $\delta_0>0$ such that for every $0<\alpha<n$, $0\leq \beta <n+1$ and $0<\delta \leq \delta_0$ satisfying $\alpha+\beta+\delta \leq 2l-1$, there exist $\epsilon\in ]0,1]$ and $c>0$ such that the following inequalities hold on $\Omega \cap \overline{U} \cap \{\left \lvert \varphi \right \rvert \leq \epsilon\}$: 

\begin{equation}
\label{SPSH}
\Delta_{w'}\left(\frac{\lvert \nabla u \rvert_w ^2}{(-\varphi)^\beta}-c(-\varphi)^\alpha\right)>0,
\end{equation}

\begin{equation}
\label{PourPMax}
\Delta_{w'}\left(\frac{\lvert \nabla u \rvert_w ^2}{(-\varphi)^\beta}-c(-\varphi)^\alpha\right)>\frac{n+1+n\beta-\beta^2}{2} \left( \frac{\lvert \nabla u \rvert_w ^2}{(-\varphi)^\beta}-c(-\varphi)^\alpha \right).
\end{equation}

Inequality \mref{SPSH} implies that the function $f:=\frac{\lvert \nabla u\rvert_w ^2}{(-\varphi)^\beta}-c(-\varphi)^\alpha$ cannot achieve its maximum on $\Omega\cap \overline{U}\cap \{ \left \lvert \varphi \right \rvert \leq \epsilon\}$, provided it is bounded from above on the set $D_{\epsilon}:= \Omega \cap U\cap \{ \left\lvert\varphi\right\rvert < \epsilon \}$. Hence we can find a sequence $\left(z_i\right)_{i\in \mathbb{N}} \in D_\epsilon^\mathbb{N}$ such that $\displaystyle \lim_{i\to +\infty}f\left(z'_i\right)=\sup_{D_\epsilon} f$ and $d_{w'}\left(z_i,\partial D_\epsilon\right)\underset{z \to + \infty}{\longrightarrow} +\infty$. Note that this implies that there exists a positive number $R$ and an integer $i_0\in \mathbb{N}$ such that for every $i\geq i_0$ we have $d_{w'}\left(z_i,\partial D_\epsilon \right)\geq R$. 
\\The last step to conclude is to apply the local maximum principle due to J. Bland (see Lemma II in \cite{Bla1}) and use inequation \mref{PourPMax}. For completeness, we recall the local maximum principle in a version that fits our situation: 

\begin{lem}
\label{LMAxP}
Let $\Omega \subset \mathbb{C}^n$ be a domain. Assume that there exists a Kähler-Einstein metric induced by a potential $w'$ on $\Omega$. Let $D\subset \Omega$ be a domain. Let $f\in \mathcal{C}^{2}\left(D\right)$ bounded from above. If there exists a sequence $\left(z_i\right)_{i\in \mathbb{N}}\in D^\mathbb{N}$ such that $\displaystyle \lim_{i\to +\infty}f\left(z'_i\right)=\sup_D f$ and there exists $R>0$ such that for every integer $i$, $d_{w'}\left(z_i,\partial D\right)\geq R$, then there exists an other sequence $\left(z'_i\right)_{i\in \mathbb{N}}\in D^\mathbb{N}$ such that 
$$\lim_{i\to +\infty}f\left(z'_i\right)=\sup_D f,\quad \limsup_{i \to + \infty} \Delta_{w'}f(z'_i)\leq 0.$$
\end{lem}

We apply Lemma \ref{LMAxP} to $\displaystyle f=\frac{\lvert \nabla u \rvert_w ^2}{(-\varphi)^\beta}-c(-\varphi)^\alpha$ with $D=D_\epsilon$ and choose the suitable constants $\alpha,\beta,\delta$ to conclude. We may argue as follows.  

\begin{enumerate}[wide=0pt]
\item
If $2n+1\leq 2l-1$, we first apply Lemma \ref{LMAxP} with $\beta=0$, $\alpha =n-\frac{\delta}{4}$ and $\delta \in \left]0,\min\left(\delta_0,4n\right)\right[$ to deduce the existence of constants $\epsilon \in ]0,1]$ and $c>0$ for which we have $\lvert \nabla u \rvert_w ^2-c(-\varphi)^{n-\frac{\delta}{4}}\leq 0$ on $\Omega \cap D_{\epsilon}$. Since $\left(-\varphi\right)<\epsilon \leq 1$ on $\Omega \cap D_\epsilon$, this directly implies: $\lvert \nabla u \rvert_w ^2-c(-\varphi)^{n-\frac{\delta}{2}}\leq 0$ on $\Omega \cap D_{\epsilon}$.
\item 
Hence we may apply Lemma \ref{LMAxP} with $\alpha =\beta = n -\frac{\delta}{2}$ and $\delta \in \left]0,\min\left(\delta_0,2n\right)\right[$ to deduce the existence of constants $\epsilon \in ]0,1]$ and $c>0$ for which $\displaystyle \frac{\lvert \nabla u \rvert_w ^2}{(-\varphi)^{n -\frac{\delta}{2}}}-c(-\varphi)^{n -\frac{\delta}{2}}\leq 0$ on $\Omega \cap D_{\epsilon}$. Again, since $\left(-\varphi\right)<\epsilon \leq 1$ on $\Omega \cap D_\epsilon$, this directly implies: $\lvert \nabla u \rvert_w ^2-c(-\varphi)^{n+1-\frac{\delta}{2}}\leq 0$ on $\Omega \cap D_{\epsilon}$. 
\item Hence we may apply once more Lemma \ref{LMAxP} with $\beta=\alpha+1=n+1-\frac{\delta}{2}$ and $\delta \in \left]0,\min\left(\delta_0,2n\right)\right[$ to deduce the existence of $c,\epsilon>0$ for which $\displaystyle \frac{\lvert \nabla u \rvert_w ^2}{(-\varphi)^{n+1 -\frac{\delta}{2}}}-c(-\varphi)^{n -\frac{\delta}{2}}\leq 0$ on $\Omega \cap D_{\epsilon}$. Finally, we directly deduce that $\lvert \nabla u \rvert_w ^2\leq c(-\varphi)^{2n+1-\delta}$ on $\Omega\cap D_\epsilon$.
\item If $2l-1<2n+1$, we can proceed likewise: first taking $\displaystyle \beta=0,\alpha=\min\left(n,l-\frac{1}{2}\right)-\frac{\delta}{8}$ with $\delta \in \left]0,\min\left(\delta_0,8\min\left(n,l-\frac{1}{2}\right)\right)\right[$, then considering $\displaystyle \alpha=\beta=\min\left(n,l-\frac{1}{2}\right)-\frac{\delta}{4}$ with $\delta \in \left]0,\min \left(\delta_0,4\min\left(n,l-\frac{1}{2}\right)\right)\right[$, and finally taking $\displaystyle \alpha=\beta =l-\frac{1}{2}-\frac{\delta}{2}$ with $\delta \in \left]0,\min \left(\delta_0,2l-1\right)\right[$.
\end{enumerate}
In both cases, we obtain the desired conclusion by letting $\delta$ tend to 0. Hence the result.
\end{proof}

In the rest of Subsection 3.1, we use Proposition \ref{ControlNab} first to derive the estimates of $u$ of order $0$ (Proposition \ref{FirstCsq}), second to derive estimates of higher order (Proposition \ref{Dpu}), and finally to obtain a regularity result for $\varphi e^{-u}$ (Proposition \ref{Reg}).

\begin{prop}
\label{FirstCsq}
Under the hypothesis and notations of Proposition \ref{ControlNab}, we have: 
\begin{enumerate}
\item 
\label{ByProduct}
For every $\gamma\in ]0,\min\left(2n+1,2l-1\right)[$, there exist positive constants $\epsilon$ and $c$ such that $\left \lvert \nabla u \right \rvert \leq c \left(-\varphi \right)^{\frac{\gamma}{2}-1}$ on the set $\Omega \cap U\cap \{ \lvert \varphi \rvert  < \epsilon \}$. In particular, if $\gamma >2$, one has $u\in \mathcal{C}^1\left(\overline{\Omega \cap U}\right)$. 

\item 
\label{CYLocBis}
For every $z \in \partial \Omega \cap U$, $\left\lvert \nabla e^{-w'} \right\rvert_z\neq 0$.

\item 
\label{utophi}
For every $\gamma \in ]0,\min(2n+1,2l-1)[$ there exist positive constants $c$ and $\epsilon$ such that $\left\lvert u \right \rvert\leq c \left(- \varphi \right)^\frac{\gamma}{2}$ on $ \Omega\cap U \cap \{\left \lvert \varphi \right \rvert <\epsilon \}$.
\end{enumerate}
\end{prop}

\begin{rem}
\label{Already}
\begin{itemize}[wide=0pt]
\item Observe that relation \mref{Lap} already gives a control on $u$. Indeed, by applying $Log\circ Det$ on both sides, using equation \mref{MA}, and simplifying both sides, we may successively obtain, on $\Omega\cap \overline{U} \cap \{\left \lvert \varphi \right \rvert \leq \epsilon \}$:  
\begin{align*}
&e^{(n+1)u-F}Det\left(w_{i\bar{j}}\right)=\left(1+O\left(\left \lvert \varphi \right \rvert^{\delta_0}\right) \right)^n Det\left(w_{i\bar{j}}\right),
\\&u=\frac{n}{n+1}Log\left(1+O\left(\left \lvert \varphi \right \rvert^{\delta_0}\right) \right)+\frac{F}{n+1}.
\end{align*}
Thus, part $(3)$ of Proposition \ref{FirstCsq} only improves the exponent $\delta_0$.
\item Part $(3)$ of Proposition \ref{FirstCsq} is exactly as in \cite{Bla1}, the only difference being that we have it for every $\gamma \in \left]0,\min\left(2n+1,2l-1\right)\right[$. We prove it a slightly different way by first proving part $(\ref{ByProduct})$ of Proposition \ref{FirstCsq}.
\end{itemize}
\end{rem}

\begin{proof}[Proof of Proposition \ref{FirstCsq}]
\begin{enumerate}[wide=0pt]
\item
We apply Proposition \ref{ControlNab}, and use Proposition \ref{EstInv} with $\psi = \eta \varphi$, $g=w$ and $U$ replaced with $U \cap \{ \lvert \varphi \rvert < \epsilon \}$. With notations of Propositions \ref{ControlNab} and \ref{EstInv}, we have $\frac{c}{\lambda}>0$. Moreover we know that $-\psi+\lvert \nabla \psi \rvert_\psi^2,\left(\frac{1}{\eta}\right)^2\in \mathcal{C}\left(\overline{U}\right)$ and are positive functions. Hence they are bounded from above, so that there exist positive constants $M_1,M_2$ such that $-\psi+\lvert \nabla \psi \rvert_\psi^2\leq M_1$ and $\left(\frac{1}{\eta}\right)^2\leq M_2$ on $\overline{U}$. Thus, we have the following on $\Omega \cap \overline{U}$: 
\begin{align*}
\left \lvert \nabla u \right \rvert ^2 &\leq \frac{1}{\lambda}\frac{-\psi+\lvert \nabla \psi \rvert_\psi^2}{\psi^2}\left \lvert \nabla u \right \rvert_g ^2\leq \frac{c}{\lambda}(-\psi+\lvert \nabla \psi \rvert_\psi^2)\left(\frac{\varphi}{\psi}\right)^2\left(-\varphi\right)^{\gamma -2},
\\&=\frac{c}{\lambda}(-\psi+\lvert \nabla \psi \rvert_\psi^2)\left(\frac{1}{\eta}\right)^2\left(-\varphi\right)^{\gamma -2},
\\&\leq \frac{c}{\lambda}M_1M_2\left(-\varphi\right)^{\gamma -2}.
\end{align*}
Therefore we obtain the conclusion by setting $c'=\sqrt{\frac{c}{\lambda}M_1M_2}$. Especially, if $\gamma>2$, then all the derivatives of $u$ of order 1 extend continuously to $\overline{\Omega \cap U}$ (and equal 0 on $\partial \Omega \cap \overline{U}$), hence $u\in \mathcal{C}^1\left(\overline{\Omega \cap U}\right)$. 

\item 
To prove part $(\ref{CYLocBis})$ of Proposition \ref{FirstCsq}, we let $l=n+1$. Then by construction $e^{-w'}=-\varphi^{(n+1)}e^{-u}$. Moreover, according to point $(1)$ of Proposition \ref{FeffAp} and to point $(1)$ of Proposition \ref{FirstCsq}, we have $\varphi^{(n+1)},u \in \mathcal{C}^1\left(\overline{\Omega \cap U}\right)$. Thus $e^{-w'}\in \mathcal{C}^1\left(\overline{\Omega \cap U}\right)$ so that we can differenciate in $\Omega\cap U$ and let $z$ tend to any point in $\partial \Omega \cap U$ to deduce
$$\lim_{z\to \partial \Omega\cap U}\left\lvert \nabla e^{-w'}\right\rvert_z=\lim_{z\to \partial \Omega\cap \overline{U}} \left \lvert \nabla \varphi^{(n+1)}\right \rvert_z \neq 0 ,$$
because of points $(2),(3)$ of Proposition \ref{FeffAp}.

\item
Fix $\gamma \in \left]0,\min\left(2n+1,2l-1\right)\right[$.
\\Let $z \in U \cap \{\lvert \varphi \rvert <\epsilon \}$. Let $z_0 \in \partial \Omega\cap U \cap \{\lvert \varphi\rvert <\epsilon \}$ such that $d(z,\partial \Omega)=\lvert z-z_0 \rvert=:s$. Set $\overrightarrow{v}:=z-z_0$. Define the following function: 
$$
\begin{array}{cccc}
f:& [0,1 ] &\longrightarrow &\mathbb{R}
\\ &t       &\longmapsto &u\left(z_0+ t\overrightarrow{v}\right).
\end{array}
$$
According to point \mref{ByProduct} of Proposition \ref{FirstCsq} we have $f \in \mathcal{C}^1\left([0,s]\right)$. 
Moreover, by the Cauchy-Schwarz inequality we have $\left \lvert f'(t)\right\rvert \leq \left \lvert \nabla u \right \rvert_{z_0+t\overrightarrow{v}} \left \lvert \overrightarrow{v} \right \rvert=s\left \lvert \nabla u \right \rvert_{z_0+t\overrightarrow{v}}$. From point $(1)$ of Remark \ref{Already} we also have $u(z_0)=0$. Using the fundamental theorem of calculus we deduce: 
\begin{align*}
\lvert u(z) \rvert =\lvert f(1)-f(0)\rvert &=\left\lvert\int_0^1 f'(t) \; dt\right\rvert,
\\&\leq s \int_0^1 \left \lvert \nabla u \right \rvert_{z_0+t\overrightarrow{v}} \; dt,
\\&\leq cs \int_0^1 \left(-\varphi\left(z_0+t\overrightarrow{v}\right)\right)^{\frac{\gamma}{2}-1}\; dt,
\\&\leq \frac{cs }{\inf_{[0,1]}  h'(t)} \int_0^1 h'(t)\left(h(t)\right)^{\frac{\gamma}{2}-1} \; dt,
\\&=\frac{2cs}{\gamma \inf_{[0,1]}  h'(t)}\int_0^1 \left(h^{\frac{\gamma}{2}}\right)'(t)\; dt,
\\&=\frac{2cs}{\gamma \inf_{[0,1]}  h'(t)}\left(-\varphi(z)\right)^\frac{\gamma}{2},
\\&\leq \frac{2cs}{\gamma \inf_{[0,1]}  h'(t)},
\end{align*}
where $h:=-\varphi\left(z_0+\cdot\overrightarrow{v}\right)\in \mathcal{C}^1\left([0,1]\right)$. According to point $(3)$ of Proposition \ref{FeffAp} we have $\displaystyle \inf_{[0,1]}  h'>0$. Hence the result. 
\end{enumerate}
\end{proof}

Proposition \ref{Dpu} is exactly as in \cite{Bla1}, the only difference being that we have the estimates for every $\gamma \in \left]0,\min\left(2n+1,2l-1\right)\right[$.

\begin{prop}
\label{Dpu}
Under the hypothesis and notations of Proposition \ref{ControlNab}, we have: for every $ \gamma \in ]0;\min(2n+1,2l-1)[$, there exist positive constants $\epsilon$ and $c$ such that for every integer $0\leq p \leq k-2l$, the following holds on $ \Omega \cap U \cap \{\left \lvert \varphi \right \rvert < \epsilon \}$: 
$$\left\lvert D^p u \right\rvert_w\leq c \left \lvert \varphi\right \rvert^\frac{\gamma}{2},$$
where $\left \lvert D^p u \right \rvert_w$ is the length of the p-th covariant derivative of $u$ with respect to $\left(w_{i\bar{j}}\right)$.
\end{prop}

\begin{proof}[Proof of Proposition \ref{Dpu}]
We fix $\gamma \in \left ]0, \min \left(2n+1,2l-1\right)\right[$ and follow line by line the proof at the beginning of page 300 in \cite{Bla1}, the only thing that changes being the range in which $\gamma$ can be choosen.
Namely, we apply $Log\circ Det$ to equation \mref{MA} to obtain the following partial differential equation of second order: 
\begin{equation}
\label{MALin}
(n+1)u-F=h_{i\bar{j}} u_{j\bar{i}},
\end{equation}
where $\displaystyle \left(h_{i\bar{j}}\right):=\left(\int_0^1\left(w+tu\right)^{i\bar{j}}\;dt\right)\in \mathcal{C}^{k-2l-2}\left(\Omega \cap U \cap \{\left \lvert \varphi \right \rvert  <\epsilon \},\mathcal{H}_n^{++}\right).$
We use equation \mref{c_1} with $\delta=0$ to deduce the existence of constants $\epsilon,c>0$ such that we have, on $\Omega \cap U \cap \{\left \lvert \varphi \right \rvert < \epsilon \}$: 
$$\frac{1}{c}\left(w^{i\bar{j}}\right)\leq \left(h_{i\bar{j}}\right)\leq c\left(w^{i\bar{j}}\right).$$
Moreover $u\in \mathcal{C}^{k-2l}\left(\Omega \cap U \cap \{\lvert \varphi \rvert < \epsilon \}\right)$, and according to Proposition \ref{FeffAp} we have $F,\frac{F}{\left(-\varphi\right)^l} \in \mathcal{C}^{k-2l-2}\left(\overline{\Omega \cap U \cap \{\lvert \varphi \rvert < \epsilon \}}\right)$. We conclude by applying Schauder theory.
\end{proof}

In particular, we deduce the following, exactly as was done in \cite{Bla1}: 
\begin{prop}
\label{Reg}
Under the notations and hypothesis of Proposition \ref{ControlNab}, for every number $\gamma \in ]0,\min\left(2n+1,2l-1\right)[$ and for every $0\leq \delta < \frac{\gamma}{2}-\left\lfloor\frac{\gamma}{2}\right\rfloor$ (where $\left\lfloor\frac{\gamma}{2}\right\rfloor$ denotes the integral part of $\frac{\gamma}{2}$), we have: $u,e^{-u} \in \mathcal{C}^{\left\lfloor\frac{\gamma}{2}\right\rfloor+\delta}\left(\overline{\Omega \cap U}\right)$. Moreover, if $\gamma > 2$, we have: $\varphi e^{-u}\in \mathcal{C}^{\left\lfloor\frac{\gamma}{2}\right\rfloor+1+\delta}\left(\overline{\Omega \cap U}\right)$. 

\end{prop}

\begin{proof}[Proof of Proposition \ref{Reg}]
This is exactly as in \cite{Bla1} (or \cite{CY} for a global version). Observe that since $k-2l\geq 3n+6-2(n+1)\geq n+2 \geq \frac{\gamma}{2}$, $u\in \mathcal{C}^{n+2}\left(\Omega \cap U\right)$ and $\varphi \in \mathcal{C}^{n+2}\left(\overline{U}\right)$ (see Proposition \ref{FeffAp}), it is enough to prove the existence of a positive constant $\epsilon$ such that for every $0\leq \delta < \frac{\gamma}{2}-\left\lfloor\frac{\gamma}{2}\right\rfloor$, one has $u\in \mathcal{C}^{\left\lfloor\frac{\gamma}{2}\right\rfloor+ \delta}\left(\overline{\Omega \cap U \cap \{ \lvert\varphi\rvert < \epsilon\}}\right)$ and $\varphi e^{-u}\in \mathcal{C}^{\left\lfloor\frac{\gamma}{2}\right\rfloor+1+\delta}\left(\overline{\Omega \cap U \cap \{ \lvert\varphi\rvert < \epsilon\}}\right)$.
\\Let $\gamma \in ]0,\min\left(2n+1,2l-1\right)[$. According to Proposition \ref{Dpu}, there exist positive constants $\epsilon$ and $c$ such that for every integer $0\leq p \leq k-2l$, the following holds on $ \Omega \cap U \cap \{\left \lvert \varphi \right \rvert < \epsilon \}$: 
$$\left\lvert D^p u \right\rvert_w\leq c\left \lvert \varphi\right \rvert^\frac{\gamma}{2}.$$
Moreover, according to Proposition \ref{EstInv}, there exist positive constants $\lambda\leq \Lambda$ such that the following holds on $\Omega \cap U$: 
$$\lambda\left(\frac{-\psi}{-\varphi}\right)\frac{-\psi}{-\psi +\left \lvert \nabla \psi \right \rvert ^2}I\leq \left(\displaystyle \frac{w^{i\bar{j}}}{-\varphi}\right) \leq \Lambda\left(\frac{-\psi}{-\varphi}\right) I.$$
Since $\left(\frac{-\psi}{-\varphi}\right) \in \mathcal{C}\left(\overline{U}\right)$ is a positive function (see Proposition \ref{FeffAp}) and $\overline{U}$ is a compact set, we deduce that there exist positive constants $M$ and $M'$ such that the following holds on $\Omega \cap U$: 
$$\lambda M \frac{-\psi}{-\psi +\left \lvert \nabla \psi \right \rvert ^2}I\leq \left(\displaystyle \frac{w^{i\bar{j}}}{-\varphi}\right) \leq \Lambda M' I.$$
Together with the expression of $\left \lvert D^p u \right \rvert_w$ in terms of the derivatives of $u$ and of $w$, this implies the existence of positive constants $\epsilon$ and $c$ such that for every integer $0\leq p \leq k-2l$ and every multi-index $(i_1,j_1,\cdots,i_n,j_n)\in \mathbb{N}^{2n}$ satisfying $\sum_{k=1}^n(i_k+j_k)\leq p$, the following holds on $ \Omega \cap U \cap \{\left \lvert \varphi \right \rvert < \epsilon \}$: 
$$\left\lvert u_{i_1\overline{j_1}\cdots i_n\overline{j_n}} \right\rvert,\left\lvert \left(e^{-u}\right)_{i_1\overline{j_1}\cdots i_n \overline{j_n}} \right\rvert\leq c\left \lvert \varphi\right \rvert^{\frac{\gamma}{2}-p}.$$
$\bullet$ Let $p=\left\lfloor\frac{\gamma}{2}\right\rfloor$. Then the derivatives of $u$ of order $p$ extend continuously to $ \Omega \cap U \cap \{\left \lvert \varphi \right \rvert < \epsilon \}$ (and are equal to $0$ on $\partial{\Omega} \cap \overline{U}$), and these extensions are Hölder of exponent $\delta$ for every $0\leq \delta < \frac{\gamma}{2}-\left\lfloor\frac{\gamma}{2}\right\rfloor$. This gives the desired regularity of $u$ and $e^{-u}$.
\\ $\bullet$ According to the chain rule and the regularity of $\varphi$ and $e^{-u}$, we have the existence of a constant $c>0$ such that the following holds on $\Omega \cap U \cap \{ \lvert \varphi \rvert < \epsilon \}$: 
$$\left\lvert \left(\varphi e^{-u}\right)_{i_1 \overline{j_1}\cdots i_n \overline{j_n}} - \varphi _{i_1 \overline{j_1} \cdots i_n\overline{j_n}}e^{-u}\right\rvert\leq c\left \lvert \varphi\right \rvert^{\frac{\gamma}{2}-(p-1)}.$$
Moreover, we have $\displaystyle \varphi _{i_1\overline{j_1}\cdots i_n \overline{j_n}}e^{-u} \in \mathcal{C}^1\left(\overline{\Omega \cap U \cap\{\lvert \varphi \rvert < \epsilon \}} \right)\subset \cap_{0\leq \delta\leq 1} \mathcal{C}^{\delta}\left(\overline{\Omega \cap U \cap\{\lvert \varphi \rvert < \epsilon \}} \right)$ because we assume that $\frac{\gamma}{2}> 1$. Let $p=\left\lfloor\frac{\gamma}{2}\right\rfloor+1$. Then the derivatives of $\varphi e^{-u}$ of order $p$ extend continuously to $ \Omega \cap U \cap \{\left \lvert \varphi \right \rvert < \epsilon \}$ and these extensions are Hölder of exponent $\delta$ for every $0\leq \delta < \frac{\gamma}{2}-\left\lfloor\frac{\gamma}{2}\right\rfloor$. This gives the desired regularity of $\varphi e^{-u}$.
\end{proof}

\subsection{Proof of Theorems \ref{CYLoc} and \ref{Main}}
We deduce Theorem \ref{CYLoc} by using Proposition \ref{Reg}: 

\begin{proof}[Proof of Theorem \ref{CYLoc}]
We take $l=n+1$. Then, according to Proposition \ref{NabFOK}, the range of $\gamma$ is $]0,2n+1[$. Let $\alpha\in \left]0,1\right[$ and take $\gamma:=2n+\alpha$ so that $\displaystyle \left\lfloor \frac{\gamma}{2}\right\rfloor=n$. We apply Proposition \ref{Reg} to obtain $\varphi e^{-u} \in \mathcal{C}^{n+1 +\delta}\left(\overline{\Omega\cap U}\right)$ for every $0\leq \delta <\frac{\alpha}{2}$. Since $k-2(n+1)\geq n+2$, then $\displaystyle \frac{\varphi^{(n+1)}}{\varphi} \in \mathcal{C}^{n+2}\left(\overline{U}\right)$ by point $(7)$ of Proposition \ref{FeffAp}. We directly deduce that $\displaystyle -w'=\varphi^{(n+1)}e^{-u}=\left(\frac{\varphi^{(n+1)}}{\varphi}\right) \varphi e^{-u}\in \mathcal{C}^{n+1 +\delta}\left(\overline{\Omega\cap U}\right)$. This holds for every $0\leq \delta <\frac{\alpha}{2}<\frac{1}{2}$, hence the result.
\end{proof}

We can also prove Theorem \ref{Main}: 
\begin{proof}[Proof of Theorem \ref{Main}]
By definition, 
$$Bis_{g,z}(v,w)=\displaystyle \frac{\sum_{1\leq i,j,k,l \leq n}R_{i\bar{j}k\bar{l}}(g)v_i\overline{v_j}w_k \overline{w_l}}{\lvert v \rvert_{g,z} ^2 \lvert w \rvert_{g,z} ^2},$$
 where the curvature coefficients satisfy the following formula which follows from the definition by direct calculations: 
\begin{equation}
\begin{array}{lll}
\label{RPhi}
R_{i\bar{j}k\bar{l}}(g)=&-(g_{i\bar{j}}g_{k\bar{l}}+g_{i\bar{l}}g_{k\bar{j}})
\\
\\ &+\displaystyle\frac{1}{-\psi}\left(R_{i\bar{j}k\bar{l}}(\psi)-\frac{1}{|\nabla \psi|_\psi^2-\psi}\underbrace{\left(\psi_{ik}-\psi_{ik\bar{p}}\psi^{\bar{p}q}\psi_q\right)}_{\psi_{,ik}:=}\underbrace{\left(\psi_{\bar{j}\bar{l}}-\psi_{\bar{p}}\psi^{\bar{p}q}\psi_{q\bar{j}\bar{l}}\right)}_{\psi_{,\bar{j}\bar{l}}:=}\right).
\end{array}
\end{equation}
Therefore, if $v,w \in S(0,1)$, we have the following on $\Omega \cap U$: 
\begin{align*}
Bis_{g}(v,w)=&-\underbrace{\left(1+\frac{\left\lvert \langle v;w\rangle_{g}\right \rvert^2}{\lvert v \rvert_{g}^2 \lvert w \rvert_{g}^2}\right)}_{=:T_1(v,w)}
\\ &+\underbrace{\frac{1}{-\psi}\frac{|v|_{\psi}^2|w|_{\psi}^2}{|v|_{g}^2|w|_{g}^2}Bis_{\psi}(v,w)}_{=:T_2(v,w)}
\\ &-\underbrace{\frac{1}{-\psi}\frac{1}{|\nabla \psi|_\psi^2-\psi}\frac{\psi_{,ik}\psi_{,\bar{j}\bar{l}}v_i\overline{v_j}w_k\overline{w_l}}{|v|_{g}^2|w|_{g}^2}}_{=:T_3(v,w)}.
\end{align*}
Using the proof of Proposition \ref{Reg} with $\frac{\gamma}{2} = n+ \delta\geq 2+\delta$ for some fixed $0<\delta<\frac{1}{2}$ we have the existence of positive constants $c,\epsilon>0$ such that the following holds on $\Omega \cap U \cap \{ \lvert \varphi \rvert < \epsilon \} $ for every $1\leq i,j,k,l \leq n$: 
$$\left\lvert \psi_{i\bar{j}k\bar{l}} \right\rvert\leq c\left \lvert \varphi\right \rvert^{-1+\delta},$$
and we also have $\psi \in \mathcal{C}^3\left(\overline{\Omega \cap U \cap \{ \lvert \varphi \rvert < \epsilon \}}\right)$.
\\The rest of the proof consists of estimating $\lvert T_2(v,w) \rvert$ and $ \lvert T_3(v,w) \rvert$. This will directly follow from formulas \mref{Comp1} and \mref{Spe}.
\\$\bullet$ Using the notations of Proposition \ref{-Log(-psi)} and of the proof of Proposition \ref{-Log(-psi)}, we have $0\leq B$, hence $I\leq A$, hence $\left(\psi_{i \bar{j}}\right)=:R^2 \leq RAR=(-\psi)(g_{i\bar{j}}).$

This means that for every $v\in \mathbb{C}^n$, the following holds on $\Omega \cap U $:  
\begin{align}
\label{Comp1} |v|_{\psi}^2 &\leq (-\psi)|v|_{g}^2.
\end{align}
\\$\bullet$ Since $\overline{\Omega \cap U \{ \lvert \varphi \rvert <\epsilon \} } $ is compact and $\psi \in \mathcal{C}^2\left(\overline{\Omega\cap U \cap \{ \lvert \varphi \rvert <\epsilon \} }\right)$, we also have the existence of a positive constant $0<\lambda_-$ such that the following inequality holds on $\Omega\cap U \cap \{ \lvert \varphi \rvert <\epsilon \}  $: 
\begin{equation}
\label{Spe}
\lambda_- I\leq (\psi_{i\bar{j}}).
\end{equation}

We complete the proof as follows. According to inequality \mref{Comp1}, we have the following on $\Omega\cap U$ for every vectors $v,w \in S(0,1)$:   
$$ \frac{1}{-\psi}\frac{|v|_{\psi}^2|w|_{\psi}^2}{|v|_{g}^2|w|_{g}^2}\leq (-\psi)=(-\varphi)e^{-u}.$$
Moreover, there exists a constant $c>0$ such that for all $1\leq i,j,k,l \leq n$ we have $\left \lvert R_{i\bar{j}k\bar{l}}\left(\psi\right) \right \rvert \leq c \left \lvert \varphi \right \rvert ^{\delta -1}$ on $\Omega\cap U \cap \{ \lvert \varphi \rvert < \epsilon \}$. Hence there exists a positive constant $c>0$ such that $\left \lvert T_2(v,w)\right \rvert \leq c \lvert\varphi \rvert^\delta$ on $\Omega\cap U \cap \{ \lvert \varphi \rvert < \epsilon \}$.

Likewise, using inequalities \mref{Comp1} and \mref{Spe} we obtain, on $\Omega \cap U$ and for every vectors $v,w \in S(0,1)$: 
$$-\frac{1}{-\psi}\frac{1}{\left\lvert \nabla\psi\right \rvert_\psi^2-\psi}\frac{1}{|v|_{g}^2|w|_{g}^2}\leq \frac{(-\psi)}{\left\lvert \nabla\psi\right \rvert_\psi^2-\psi}\frac{1}{\lambda_-^2}=\frac{(-\varphi)}{\left\lvert \nabla\psi\right \rvert_\psi^2-\psi}\frac{e^{-u}}{\lambda_-^2}.$$
Note that (up to taking a smaller positive $\epsilon$) $\left\lvert \nabla \psi\right \rvert_\psi^2-\psi \in \mathcal{C} \left(\overline{\Omega \cap U \cap \{ \lvert \varphi \rvert < \epsilon \} } \right)$ and is a positive function thanks to point \mref{CYLocBis} of Proposition \ref{FirstCsq}.
Moreover, there exists a constant $c>0$ such that for all $1\leq i,j,k,l \leq n$ we have $\left \lvert \psi_{,ik}\psi_{,\bar{j}\bar{l}} \right \rvert \leq c $ on $\Omega\cap U \cap \{ \lvert \varphi \rvert < \epsilon \}$. Hence there exists a positive constant $c>0$ such that $\left \lvert T_3(v,w)\right \rvert \leq c \lvert\varphi \rvert$ on $\Omega\cap U \cap \{ \lvert \varphi \rvert < \epsilon \}$.
\\Using the triangle inequality, we deduce the existence of positive constants $\epsilon,c > 0$ such that the following inequality holds on $\Omega\cap U \cap \{ \lvert \varphi \rvert < \epsilon \}$: 
$$   \sup_{v,w \in S(0,1)}\left|Bis_{g}\left(v,w\right)+T_1(v,w)\right|\leq  c\lvert\varphi \rvert ^\delta.$$
We obtain the result since $\displaystyle \lim_{z\to q} \varphi(z)=0$ and $\delta>0$.
\end{proof}

\begin{rem}
Regardless that $n\geq 2$, the asymptotic curvature behavior \mref{AsyBis} does not persist if we remove the hypothesis of strict pseudoconvexity. 
\\For instance, if $m \in \mathbb{N}^\ast$, in the ``egg domain"
$\{(z_1,z_2)\in \mathbb{C}^2 / \lvert z_1 \rvert^2+\lvert z_2 \rvert ^{2m} <1\},$
we can easily deduce from the computations done in \cite{Bla2} that, at $q=(1,0)$: 
$$\forall v,w \in S(0,1),\quad -3+\frac{3}{2m+1}\leq \lim_{t\to 1^-}Bis_{g,(t,0)}(v,w)\leq -\frac{3}{2m+1}.$$
This differs from \mref{AsyBis} if $m\geq 2$. We also notice that the same approach as in \cite{Bla2} may be adapted to obtain the same estimates in tube domains  $\{(z_1,z_2)\in \mathbb{C}^2 / Re(z_1)+Re(z_2)^{2m} <1\}$ at $q=(1,0)$ for $m\in \mathbb{N}^\ast$.
\end{rem}

\section{Proof of Theorem \ref{BiSqueez}}
We recall the definition of the squeezing function of a domain.

\begin{defn}
\label{DefSq}
Let $\Omega \subset \mathbb{C}^n$ be a domain. For $z \in \Omega$, let  
\[ \mathcal{F}_z^\Omega:= \{f:\Omega \longrightarrow B(0,1) / \text{ $f$ is holomorphic, injective and } f(z)=0\}.\]
The \textit{squeezing funtion} of $\Omega$ at point $z$ is defined by $s^\Omega(z):=\sup\{r>0 / \exists f \in \mathcal{F}_z^\Omega , B(0,r) \subset f(\Omega)\}$.
\end{defn}
In \cite{DGZ1} the authors prove that the supremum in Definition \ref{DefSq} is achieved.

In the rest of this Section, every domain that appears possesses a unique complete Kähler-Einstein potential which is solution to Equation \mref{MAOri} with condition \mref{BVal} and we only work with this one. Moreover, given a domain $D$ with complete Kähler-Einstein potential $g$ solving Equation \mref{MAOri} with condition \mref{BVal}, we use the notations $\langle \cdot , \cdot \rangle_{z}^D, \left\lvert \cdot \right \rvert_{z}^D, Bis_{z}^D$ instead of the previous notations $\langle \cdot , \cdot \rangle_{g,z}, \left\lvert \cdot \right \rvert_{g,z}, Bis_{g,z}$ to avoid confusions.
\\We need the following Lemma, which is a direct consequence of the proof of Theorem 7.5. in \cite{CY}:
\begin{lem}
\label{Exh}
Let $D \subset \mathbb{C}^n$ be a bounded pseudoconvex domain. Let $\left(D_\nu\right)_{\nu \in \mathbb{N}}$ be an exhaustion of $D$ by strictly pseudoconvex domains with boundary of class $\mathcal{C}^\infty$. Then, up to extracting a subsequence from $\left(D_\nu\right)_{\nu \in \mathbb{N}}$, the following holds for every compact set $K\subset D$:
\[\sup_{z\in K} \sup_{v,w\in \mathbb{C}^n\setminus\{0\}} \left \lvert \left \langle v,w \right \rangle_{z}^{D_\nu}-\left\langle v, w \right \rangle_{z}^{D}\right \rvert \underset{\nu \to \infty}{\longrightarrow}0,\] 
\[\sup_{z\in K} \sup_{v,w\in \mathbb{C}^n\setminus\{0\}} \left \lvert Bis_{z}^{D_\nu}(v,w)-Bis_{z}^{D}(v,w)\right \rvert \underset{\nu \to \infty}{\longrightarrow}0.\]
\end{lem}


We prove Theorem \ref{BiSqueez}:
\begin{proof}[Proof of Theorem \ref{BiSqueez}]
Let $\left(z^{(\nu)}\right)_{\nu \in \mathbb{N}}\in \Omega^\mathbb{N}$ such that $\displaystyle \lim_{\nu \to \infty} z^{(\nu)} = q$. For $\nu \in \mathbb{N}$ let $f^{(\nu)} \in \mathcal{F}_{z^{(\nu)}}^\Omega$ such that $B\left(0,s^\Omega\left(z^{(\nu)}\right)\right)\subset f^{(\nu)}(\Omega)$, let $g^{(\nu)}:=\left(1-\frac{1}{2^{\nu +1}}\right)f^{(\nu)}$ and set $\Omega_\nu:=g^{(\nu)}\left(\Omega\right)$. Since $g^{(\nu)}$ is a biholomorphic mapping from the pseudoconvex domain $\Omega$ to $\Omega_\nu$, $\Omega_\nu$ is a bounded pseudoconvex domain. By construction of $g^{(\nu)}$, for every integer $\nu \in \mathbb{N}$ we have $\overline{\Omega_\nu} \subset B(0,1)$. Moreover we have $\displaystyle \lim_{\nu \to \infty}s^\Omega\left(z^{(\nu)}\right)=1$ hence up to taking a subsequence we may assume that $\overline{\Omega_{\nu}}\subset \Omega_{\nu +1}$. 
\\Let $\nu \in \mathbb{N}^\ast$. Since $\Omega_\nu$ is a bounded pseudoconvex domain, there exists an exhaustion of $\Omega_\nu$ by strictly pseudoconvex domains with smooth boundary, so that according to Lemma \ref{Exh} there exists a strictly pseudoconvex domain $D_\nu $ with boundary of class $\mathcal{C}^\infty$ that satisfies $\overline{\Omega_{\nu-1}} \subset D_\nu \subset \Omega_\nu$ and
\begin{align*} 
&\displaystyle \sup_{v,w\in \mathbb{C}^n\setminus\{0\}} \left \lvert \left(\frac{\left \lvert \left \langle v,w \right \rangle_{0}^{\Omega_\nu}\right\rvert}{{\left \lvert  v \right\rvert_{0}^{\Omega_\nu}}{\left \lvert  w \right\rvert_{0}^{\Omega_\nu}}}\right)^2-\left(\frac{\left \lvert \left \langle v,w \right \rangle_{0}^{D_\nu}\right\rvert}{{\left \lvert  v \right\rvert_{0}^{D_\nu}}{\left \lvert  w \right\rvert_{0}^{D_\nu}}}\right)^2\right \rvert \leq \frac{1}{2^{\nu}},
\\&\displaystyle \sup_{v,w\in \mathbb{C}^n\setminus\{0\}} \left \lvert Bis_0^{\Omega_{\nu}}(v,w)-Bis_0^{D_{\nu}}(v,w)\right \rvert \leq \frac{1}{2^{\nu}}.
\end{align*}
Moreover, since each $g^{(\nu)}$ is holomorphic and injective, the linear map $\partial g^{(\nu)}_{z^{(\nu)}}$ is invertible, hence:
\begin{eqnarray} 
&\label{SPCVProcheMet}\displaystyle \sup_{v,w\in \mathbb{C}^n\setminus\{0\}} \left \lvert \left(\frac{\left \lvert \left \langle  \partial g^{(\nu)}_{z^{(\nu)}}(v),\partial g^{(\nu)}_{z^{(\nu)}}(w) \right \rangle_{0}^{\Omega_\nu}\right\rvert}{{\left \lvert  \partial g^{(\nu)}_{z^{(\nu)}}(v) \right\rvert_{0}^{\Omega_\nu}}{\left \lvert  \partial g^{(\nu)}_{z^{(\nu)}}(w) \right\rvert_{0}^{\Omega_\nu}}}\right)^2-\left(\frac{\left \lvert \left \langle  \partial g^{(\nu)}_{z^{(\nu)}}(v),\partial g^{(\nu)}_{z^{(\nu)}}(w) \right \rangle_{0}^{D_\nu}\right\rvert}{{\left \lvert  \partial g^{(\nu)}_{z^{(\nu)}}(v) \right\rvert_{0}^{D_\nu}}{\left \lvert  \partial g^{(\nu)}_{z^{(\nu)}}(w) \right\rvert_{0}^{D_\nu}}}\right)^2\right \rvert \leq \frac{1}{2^{\nu}},
\\&\label{SPCVProcheBis}\displaystyle \sup_{v,w\in \mathbb{C}^n\setminus\{0\}} \left \lvert Bis_0^{\Omega_{\nu}}(\partial g^{(\nu)}_{z^{(\nu)}}(v),\partial g^{(\nu)}_{z^{(\nu)}}(w))-Bis_0^{D_{\nu}}(\partial g^{(\nu)}_{z^{(\nu)}}(v),\partial g^{(\nu)}_{z^{(\nu)}}(w))\right \rvert \leq \frac{1}{2^{\nu}}.
\end{eqnarray}
Because of the property $\overline{\Omega_{\nu}} \subset D_{\nu+1} \subset \Omega_{\nu+1}$ for every $\nu \in \mathbb{N}$, the sequence $\left(D_\nu\right)_{\nu \in \mathbb{N}}$ is an increasing sequence of strictly pseudoconvex domains with boundary of class $\mathcal{C}^\infty$. Since $\displaystyle \lim _{\nu \to \infty}s^\Omega\left(z^{(\nu)}\right)=1$ we have $\cup_{\nu \in \mathbb{N}}D_\nu =B(0,1)$, that is $\left(D_\nu\right)_{\nu \in \mathbb{N}}$ is an exhaustion of the unit ball by strictly pseudoconvex domains with boundary of class $\mathcal{C}^\infty$. Therefore according to Lemma \ref{Exh} we deduce the following up to extracting a subsequence from $\left(D_\nu\right)_{\nu \in \mathbb{N}}$:
\[\displaystyle \sup_{v,w\in \mathbb{C}^n\setminus\{0\}} \left \lvert \left(\frac{\left \lvert \left \langle v,w \right \rangle_{0}^{D_\nu}\right\rvert}{{\left \lvert  v \right\rvert_{0}^{D_\nu}}{\left \lvert  w \right\rvert_{0}^{D_\nu}}}\right)^2-\left(\frac{\left \lvert \left \langle v,w \right \rangle_{0}^{B(0,1)}\right\rvert}{{\left \lvert  v \right\rvert_{0}^{B(0,1)}}{\left \lvert  w \right\rvert_{0}^{B(0,1)}}}\right)^2\right \rvert \underset{\nu \to \infty}{\longrightarrow}0,\]
\[ \sup_{v,w\in \mathbb{C}^n\setminus\{0\}} \left \lvert Bis_{0}^{D_{\nu}}(v,w)-Bis_{0}^{B(0,1)}(v,w)\right \rvert \underset{\nu \to \infty}{\longrightarrow}0.\]
Moreover, since each $g^{(\nu)}$ is holomorphic and injective, the linear map $\partial g^{(\nu)}_{z^{(\nu)}}$ is invertible, hence:
\begin{eqnarray} 
&\label{CVBouleMet}
\displaystyle \sup_{v,w\in \mathbb{C}^n\setminus\{0\}} \left \lvert \left(\frac{\left \lvert \left \langle \partial g^{(\nu)}_{z^{(\nu)}}(v),\partial g^{(\nu)}_{z^{(\nu)}}(w) \right \rangle_{0}^{D_\nu}\right\rvert}{{\left \lvert  \partial g^{(\nu)}_{z^{(\nu)}}(v) \right\rvert_{0}^{D_\nu}}{\left \lvert  \partial g^{(\nu)}_{z^{(\nu)}}(w) \right\rvert_{0}^{D_\nu}}}\right)^2-\left(\frac{\left \lvert \left \langle \partial g^{(\nu)}_{z^{(\nu)}}(v),\partial g^{(\nu)}_{z^{(\nu)}}(w) \right \rangle_{0}^{B(0,1)}\right\rvert}{{\left \lvert  \partial g^{(\nu)}_{z^{(\nu)}}(v) \right\rvert_{0}^{B(0,1)}}{\left \lvert  \partial g^{(\nu)}_{z^{(\nu)}}(w) \right\rvert_{0}^{B(0,1)}}}\right)^2\right \rvert \underset{\nu \to \infty}{\longrightarrow}0,
\\&\label{CVBouleBis} \sup_{v,w\in \mathbb{C}^n\setminus\{0\}} \left \lvert Bis_0^{D_{\nu}}\left(\partial g^{(\nu)}_{z^{(\nu)}}(v),\partial g^{(\nu)}_{z^{(\nu)}}(w)\right)-Bis_0^{B(0,1)}\left(\partial g^{(\nu)}_{z^{(\nu)}}(v),\partial g^{(\nu)}_{z^{(\nu)}}(w)\right)\right \rvert\underset{\nu \to \infty}{\longrightarrow}0.
\end{eqnarray}
Using triangle inequality we obtain for every integer $\nu \in \mathbb{N}$ and every vectors $v,w\in \mathbb{C}^n\setminus \{0\}$:
\begin{align*}
&\left \lvert Bis_{z^{(\nu)}}^\Omega\left(v,w\right)+1+ \left(\frac{\left \lvert \left \langle v,w \right \rangle_{z^{(\nu)}}^{\Omega}\right\rvert}{{\left \lvert  v \right\rvert_{z^{(\nu)}}^{\Omega}}{\left \lvert  w \right\rvert_{z^{(\nu)}}^{\Omega}}}\right)^2\right\rvert
\\ & = \left \lvert Bis_0^{\Omega_{\nu}}\left(\partial g^{(\nu)}_{z^{(\nu)}}(v),\partial g^{(\nu)}_{z^{(\nu)}}(w)\right)+1+ \left(\frac{\left \lvert \left \langle \partial g^{(\nu)}_{z^{(\nu)}}(v),\partial g^{(\nu)}_{z^{(\nu)}}(w) \right \rangle_{0}^{\Omega_\nu}\right\rvert}{{\left \lvert  \partial g^{(\nu)}_{z^{(\nu)}}(v) \right\rvert_{0}^{\Omega_\nu}}{\left \lvert  \partial g^{(\nu)}_{z^{(\nu)}}(w) \right\rvert_{0}^{\Omega_\nu}}}\right)^2 \right \rvert 
\\ &\leq \left \lvert Bis_0^{\Omega_{\nu}}\left(\partial g^{(\nu)}_{z^{(\nu)}}(v),\partial g^{(\nu)}_{z^{(\nu)}}(w)\right)-Bis_0^{D_{\nu}}\left(\partial g^{(\nu)}_{z^{(\nu)}}(v),\partial g^{(\nu)}_{z^{(\nu)}}(w)\right) \right \rvert 
\\&+\left \lvert Bis_0^{D_{\nu}}\left(\partial g^{(\nu)}_{z^{(\nu)}}(v),\partial g^{(\nu)}_{z^{(\nu)}}(w)\right)-Bis_0^{B(0,1)}\left(\partial g^{(\nu)}_{z^{(\nu)}}(v),\partial g^{(\nu)}_{z^{(\nu)}}(w)\right) \right \rvert
\\& +\left \lvert Bis_0^{B(0,1)}\left(\partial g^{(\nu)}_{z^{(\nu)}}(v),\partial g^{(\nu)}_{z^{(\nu)}}(w)\right)+1+ \left(\frac{\left \lvert \left \langle \partial g^{(\nu)}_{z^{(\nu)}}(v),\partial g^{(\nu)}_{z^{(\nu)}}(w) \right \rangle_{0}^{D_\nu}\right\rvert}{{\left \lvert  \partial g^{(\nu)}_{z^{(\nu)}}(v) \right\rvert_{0}^{D_\nu}}{\left \lvert  \partial g^{(\nu)}_{z^{(\nu)}}(w) \right\rvert_{0}^{D_\nu}}}\right)^2 \right \rvert
\\& +\left \lvert \left(\frac{\left \lvert \left \langle  \partial g^{(\nu)}_{z^{(\nu)}}(v),\partial g^{(\nu)}_{z^{(\nu)}}(w) \right \rangle_{0}^{\Omega_\nu}\right\rvert}{{\left \lvert  \partial g^{(\nu)}_{z^{(\nu)}}(v) \right\rvert_{0}^{\Omega_\nu}}{\left \lvert  \partial g^{(\nu)}_{z^{(\nu)}}(w) \right\rvert_{0}^{\Omega_\nu}}}\right)^2-\left(\frac{\left \lvert \left \langle  \partial g^{(\nu)}_{z^{(\nu)}}(v),\partial g^{(\nu)}_{z^{(\nu)}}(w) \right \rangle_{0}^{D_\nu}\right\rvert}{{\left \lvert  \partial g^{(\nu)}_{z^{(\nu)}}(v) \right\rvert_{0}^{D_\nu}}{\left \lvert  \partial g^{(\nu)}_{z^{(\nu)}}(w) \right\rvert_{0}^{D_\nu}}}\right)^2\right \rvert
\\&\underset{\nu \to \infty}{\longrightarrow}0.
\end{align*}
From condition \mref{SPCVProcheBis}, respectively condition \mref{CVBouleBis}, condition \mref{SPCVProcheMet}, the first term of the right hand side, respectively the second, the fourth, tends to $0$ as $\nu$ tends to $+\infty$. Moreover the Kähler-Einstein metric we work with satisfies $Bis_0^{B(0,1)}(v,w)=-1-\left(\frac{\left \lvert \langle v, w \rangle_0^{B(0,1)} \right \rvert}{\left \lvert v \right \rvert _0^{B(0,1)}\left \lvert w \right \rvert _0^{B(0,1)}}\right)^2$ for every vectors $v,w\in \mathbb{C}^2\setminus\{0\}$. We combine this remark with condition \mref{CVBouleMet} to deduce that the third term of the right hand side tends to $0$ as $\nu$ tends to $+\infty$. Hence the result.
\end{proof}
\bibliographystyle{plain}
\bibliography{Biblio} 

\begin{bibdiv}
\begin{biblist}

\bib{Bla1}{article}{
      author={{Bland}, John~S.},
       title={{Local boundary regularity of the canonical Einstein-K\"ahler
  metric on pseudoconvex domains.}},
    language={English},
        date={1983},
        ISSN={0025-5831; 1432-1807/e},
     journal={{Math. Ann.}},
      volume={263},
       pages={289\ndash 301},
}

\bib{Bla2}{article}{
      author={{Bland}, John~S.},
       title={{The Einstein-K\"ahler metric on $\{\vert z\vert \sp 2+\vert
  w\vert \sp{2p}<1\}$.}},
    language={English},
        date={1986},
        ISSN={0026-2285},
     journal={{Mich. Math. J.}},
      volume={33},
       pages={209\ndash 220},
}

\bib{CY}{article}{
      author={{Cheng}, Shiu-Yuen},
      author={{Yau}, Shing-Tung},
       title={{On the existence of a complete K\"ahler metric on non-compact
  complex manifolds and the regularity of Fefferman's equation.}},
    language={English},
        date={1980},
        ISSN={0010-3640; 1097-0312/e},
     journal={{Commun. Pure Appl. Math.}},
      volume={33},
       pages={507\ndash 544},
}

\bib{DGZ1}{article}{
      author={{Deng}, Fusheng},
      author={{Guan}, Qian},
      author={{Zhang}, Liyou},
       title={{Some properties of squeezing functions on bounded domains.}},
    language={English},
        date={2012},
        ISSN={0030-8730},
     journal={{Pac. J. Math.}},
      volume={257},
      number={2},
       pages={319\ndash 341},
}

\bib{Fef2}{article}{
      author={{Fefferman}, Charles~L.},
       title={{Monge-Amp\`ere equations, the Bergman kernel, and geometry of
  pseudoconvex domains.}},
    language={English},
        date={1976},
        ISSN={0003-486X; 1939-8980/e},
     journal={{Ann. Math. (2)}},
      volume={103},
       pages={395\ndash 416},
}

\bib{FoWo}{article}{
      author={Fornæss, John~Erik},
      author={Wold, Erlend~Fornæss},
       title={A non-strictly pseudoconvex domain for which the squeezing
  function tends to one towards the boundary},
        date={2016},
      eprint={arXiv:1611.04464},
}

\bib{Isa1}{article}{
      author={{Isaev}, A.V.},
       title={{K\"ahler-Einstein metric on Reinhardt domains.}},
    language={English},
        date={1995},
        ISSN={1050-6926; 1559-002X/e},
     journal={{J. Geom. Anal.}},
      volume={5},
      number={2},
       pages={237\ndash 254},
}

\bib{JooKi}{article}{
      author={Joo, Seungro},
      author={Kim, Kang-Tae},
       title={On boundary points at which the squeezing function tends to one},
        date={2017},
        ISSN={1559-002X},
     journal={The Journal of Geometric Analysis},
}

\bib{KZ}{article}{
      author={{Kim}, Kang-Tae},
      author={{Zhang}, Liyou},
       title={{On the uniform squeezing property of bounded convex domains in
  $\mathbb{C}^n$.}},
    language={English},
        date={2016},
        ISSN={0030-8730},
     journal={{Pac. J. Math.}},
      volume={282},
      number={2},
       pages={341\ndash 358},
}

\bib{LMbobe}{article}{
      author={{Lee}, John},
      author={{Melrose}, Richard},
       title={{Boundary behaviour of the complex Monge-Amp\`ere equation.}},
    language={English},
        date={1982},
        ISSN={0001-5962; 1871-2509/e},
     journal={{Acta Math.}},
      volume={148},
       pages={159\ndash 162},
}

\bib{MY}{misc}{
      author={{Mok}, Ngaiming},
      author={{Yau}, Shing-Tung},
       title={{Completeness of the K\"ahler-Einstein metric on bounded domains
  and the characterization of domains of holomorphy by curvature conditions.}},
    language={English},
         how={{The mathematical heritage of Henri Poincare, Proc. Symp. Pure
  Math. 39, Part 1, Bloomington/Indiana 1980, 41-59 (1983).}},
        date={1983},
}

\bib{Yeu}{article}{
      author={{Yeung}, Sai-Kee},
       title={{Geometry of domains with the uniform squeezing property.}},
    language={English},
        date={2009},
        ISSN={0001-8708},
     journal={{Adv. Math.}},
      volume={221},
      number={2},
       pages={547\ndash 569},
}

\end{biblist}
\end{bibdiv}
\end{document}